\documentclass{amsart} 

\usepackage{amssymb}
\usepackage{stmaryrd}
\usepackage{amsmath}
\usepackage{amsthm}
\usepackage{braket}

\usepackage[shortlabels]{enumitem}

\usepackage[numbers]{natbib}

\usepackage{tikz-cd}

\tikzset{%
  symbol/.style={
    draw=none,
    every to/.append style={
      edge node={node [sloped, allow upside down, auto=false]{$#1$}}
    },
  },
}
\usepackage[english]{babel}
\usepackage{csquotes}
\usepackage{verbatim}

\usepackage[all]{xy}

\usepackage{multicol}

\usepackage{lmodern}
\usepackage[T1]{fontenc}

\usepackage{xr-hyper}
\usepackage[hyperindex,breaklinks]{hyperref}
\usepackage{cleveref}
\newtheorem{theorem}{Theorem}[section]

\newtheorem{lemma}[theorem]{Lemma}
\newtheorem{cor}[theorem]{Corollary}

\theoremstyle{definition}
\newtheorem{definition}[theorem]{Definition}

\newtheorem{remark}[theorem]{Remark}

\numberwithin{equation}{subsection}

\def\reflgood{\mathop{\mathrm{refl-good}}\nolimits}
\def\refllarge{\mathop{\mathrm{refl-large}}\nolimits}
\def\good{\mathop{\mathrm{good}}\nolimits}
\def\Hsplit{\mathop{\mathrm{H-split}}\nolimits}
\def\split{\mathop{\mathrm{split}}\nolimits}
\def\lim{\mathop{\mathrm{lim}}\nolimits}

\def\Spec{\mathop{\mathrm{Spec}}}
\def\Hom{\mathop{\mathrm{Hom}}\nolimits}

\def\gl{\mathop{\mathrm{Gl}}\nolimits}
\def\pr{\mathop{\mathrm{pr}}\nolimits}

\def\supp{\mathop{\mathrm{Supp}}\nolimits}

\def\id{\mathop{\mathrm{id}}\nolimits}

\def\supp{\mathop{\mathrm{supp}}\nolimits}

\def\deg{\mathop{\mathrm{deg}}\nolimits}

\def\det{\mathop{\mathrm{det}}\nolimits}
\def\supp{\mathop{\mathrm{supp}}\nolimits}
\def\dim{\mathop{\mathrm{dim}}\nolimits}
\def\rk{\mathop{\mathrm{rk}}\nolimits}

\def\lim{\mathop{\mathrm{lim}}\nolimits}
\def\et{\mathop{\text{\'et}}\nolimits}
\def\Aut{\mathop{\mathrm{Aut}}\nolimits}
\def\im{\mathop{\mathrm{im}}\nolimits}
\def\tr{\mathop{\mathrm{tr}}\nolimits}
\def\refl{\mathop{\mathrm{Refl}}\nolimits}
\def\coh{\mathop{\mathrm{Coh}}\nolimits}
\def\id{\mathop{\mathrm{id}}\nolimits}
\def\diag{\Delta}
\def\gal{\mathop{\mathrm{Gal}}\nolimits}
\def\Sym{\mathop{\mathrm{Sym}}\nolimits}
\def\FGss{\mathop{F\mathrm{-G-ss}}\nolimits}

\usepackage{todonotes}

\makeatletter 
\def\l@subsection{\@tocline{2}{0pt}{1pc}{5pc}{}} \def\l@subsection{\@tocline{2}{0pt}{2pc}{6pc}{}} 
\makeatother
\title[Stratifying Higgs moduli and the Hitchin morphism]{Stratifying moduli spaces of Higgs bundles and the Hitchin morphism}
\author{Aryaman Patel}
\address{(A. Patel) AG Lazi\'c, Universit\"at des Saarlandes, 66123 Saarbr\"ucken, Germany}
\email{\href{mailto:aryaman.patel@math.uni-sb.de}{aryaman.patel@math.uni-sb.de}}
\author{Dario Weißmann}
\address{(D. Wei{\ss}mann) Instytut Matematyczny Polskiej Akademii Nauk, ul. Śniadeckich 8, 00-656 Warszawa, Poland}
\email{\href{mailto: dario.weissmann@posteo.de}{dario.weissmann@posteo.de}}

\date{\today}

\newcommand{\darionew}[1]{{\color{violet} #1}}

\begin{document}

\begin{abstract}
We study the behavior of slope-stability of reflexive twisted
sheaves over a normal projective variety $X$
under pullback along a cover.
Slope-stability is always preserved 
if the cover does not factor via a quasi-\'etale cover.
Fixing the rank,
there is one quasi-\'etale cover that checks whether
a twisted sheaf remains slope-stable on all Galois covers, 
yielding a stratification of the moduli space of slope-stable
Higgs-bundles.
As an application, we determine the image of the Hitchin morphism restricted to the 
smallest closed stratum of the Dolbeault moduli space when
$X$ is smooth.
This allows us to
determine the image of the Hitchin morphism from the Dolbeault moduli space when
$X$ is a hyperelliptic or abelian variety in characteristic $p\ge0$. In particular, we show that Chen-Ng\^o's conjecture holds for hyperelliptic varieties in characteristic $0$. 
\end{abstract}

\maketitle

\section{Introduction}
Stability of vector bundles is a fundamental property
used in the construction of their moduli spaces.
In recent years, the question of when the pullback
of a slope-stable vector bundle remains slope-stable has been raised. 
For normal projective varieties, this was answered in \cite{bdp}:
a finite separable morphism $Y\to X$ preserves the
slope-stability of vector bundles if $Y\to X$ is genuinely ramified, 
that is, it does not factor via a non-trivial \'etale cover of $X$.
In general, pullback along a finite \'etale cover does not preserve stability.
However, for a fixed rank, the property to remain stable on \emph{all}
\'etale covers is an open property in characteristic $0$, see \cite{fun}.
In positive characteristic $p$\darionew{,}
one has to exclude the case that $p$ divides the degree
of \'etale Galois covers in the discussion.

In this paper, we answer an analogous question for stable reflexive sheaves $\mathcal{V}$
under reflexive pullback along a finite separable morphism $Y\to X$ 
of normal projective varieties.
We also add a twisted endomorphism $\theta: \mathcal{V}\to \mathcal{V}\otimes F$,
where $F$ is a vector bundle, to the discussion.
The results follow a similar structure as in the vector bundle case,
but this time quasi-\'etale covers play the role of \'etale covers, 
that is, finite separable covers that are \'etale over an open subvariety whose complement has codimension at least $2$.
In the case that $X$ is smooth and $\theta=0$,
pullbacks of stable reflexive sheaves were also independently considered in 
\cite[Section 3.1]{f-divded-normal-langer}.

\begin{theorem}[Theorem \ref{thm:bdp analogue stable}]\label{main1}
    Let $\pi:Y\to X$ be a
    finite separable
    morphism
    of normal projective varieties that does not factor through a non-trivial
    finite \'etale morphism, and let $F$ be a vector bundle on $X$.
    
    Then an $F$-twisted vector bundle $(\mathcal{V},\theta)$ on $X$ is slope-stable 
    iff $(\pi^{\ast}\mathcal{V},\pi^{\ast}\theta)$ is slope-stable
    as a $\pi^{\ast}F$-twisted bundle on $Y$.
   
    Furthermore, if $\pi$ does not factor through a non-trivial finite 
    quasi-\'etale morphism,
    then an $F$-twisted reflexive sheaf $(\mathcal{V},\theta)$ on $X$ 
    is slope-stable if and only if $(\pi^{[\ast]}\mathcal{V},\pi^{[\ast]}\theta)$ 
    is slope-stable as a $\pi^{\ast}F$-twisted reflexive sheaf on $Y$,
    where $(-)^{[*]}$ denotes reflexive pullback. 
\end{theorem}

Thus, to understand slope-stability under reflexive pullback by a general finite
separable morphism, 
it remains to study quasi-\'etale covers.
Fixing the rank of the reflexive sheaves, we obtain
that it suffices to check slope-stability on one quasi-\'etale
cover instead of all quasi-\'etale covers: 

\begin{theorem}[Theorem \ref{thm:good cover}, char $0$]\label{main2}
    Let $X$ be a normal projective variety, let $r\geq 2$ be an integer and let $F$ be a vector bundle on $X$. 
    
    Then there exists a quasi-\'etale Galois cover
    $\pi_{r,\reflgood}:X_{r,\reflgood}\to X$ called the \emph{good cover}, such that any slope-stable $F$-twisted 
    reflexive sheaf $\mathcal{V}$ of rank $r$ on $X$ remains slope-stable after pullback to all covers $Y\to X$ 
    if and only if 
    $\pi_{r,\reflgood}^{[\ast]}\mathcal{V}$ is slope-stable as an 
    $\pi_{r,\reflgood}^\ast F$-twisted reflexive sheaf.
\end{theorem}

Given an \'etale Galois cover $\pi:Y\to X$ and a vector bundle $F$ on $X$,
any stable $F$-twisted bundle $\mathcal{V}$ of rank $\ge2$ on $X$ 
decomposes as a direct sum of stable $F_{\mid_Y}$-twisted bundles on $Y$
such that the Galois group of $\pi$ acts transitively on the isomorphism classes 
of the stable direct summands.
By transitivity, all 
stable direct summands have the same rank, 
called the \emph{decomposition type} of $\mathcal{V}$ with respect to $\pi$ 
(Definition \ref{dec type}). 

Iterating the good cover, we obtain a cover $\pi_{r,\split}:X_{r,\mathrm{split}}\to X$ 
called the \emph{split cover}, such that the decomposition type remains unchanged 
with respect to all covers dominating the split cover.
In particular, we obtain a stratification of the moduli spaces $M^{P,F-s}_X$ 
of slope-stable $F$-twisted Higgs bundles with fixed Hilbert polynomial $P$,
with respect to the decomposition type.
This stratification is independent of the choice of cover 
for all covers that dominate the split cover $X_{r,\mathrm{split}}\to X$. 

We denote by $M^{P,\FGss}_X$ the moduli space of Gieseker semistable
torsion-free Higgs sheaves on $X$ with Hilbert polynomial $P$ 
and consider the
Hitchin morphism $h^r:M^{P,\FGss}_X\to A^{r,F}_X$ which assigns to
an $F$-twisted Higgs sheaf $(\mathcal{V},\theta)$ 
the characteristic polynomial of $\theta$,
which is a point in the affine space
$A^{r,F}_X:=\bigoplus_{i=1}^rH^0(X,\mathrm{Sym}^iF)$,
called the \emph{Hitchin base}. 

Given an \'etale cover $Y\to X$, we denote by $(1)^{\oplus r}_{Y\to X}$ the subset of the Hitchin base consisting of polynomials that decompose completely, i.e., as a product of $r$ degree $1$ polynomials, after pullback to $Y$.
Let $Z^{\FGss,P}(1)$ denote the subset of $F$-twisted Higgs bundles in
$M^{P,\FGss}_X$ that split as a direct sum of $\pi^*_{r,\split}F$-Higgs bundles of rank $1$ after pullback to the split cover $X_{r,\split}$, which we call the \emph{semistable $1$-stratum} of $M^{P,\FGss}_X$. Let $P_0$ denote the Hilbert polynomial of the trivial bundle of rank $r$. Then these strata are related as follows.

\begin{theorem}[Lemma \ref{H-split} + Theorem \ref{1-stratum image}, char $0$]\label{1stratumimage}
Let $X$ be a normal projective variety and let $F$ be a vector bundle on $X$.
Then there exists an \'etale Galois cover
$X_{r,\mathrm{H-split}}\to X$, such that for every \'etale
Galois cover $Y\to X$ dominating $X_{r,\mathrm{H-split}}\to X$, we have
$(1)^{\oplus r}_{Y\to X}=(1)^{\oplus r}_{X_{r,\mathrm{H-split}}\to X}$.

Moreover, if $X$ is smooth, then the set-theoretic image
of the semistable $1$-stratum
$Z^{\FGss,P_0}(1)\subset M_X^{P_0,\FGss}$ is the subset
$(1)^{\oplus r}_{X_{r,\Hsplit}\to X}\subset A^{r,F}_X$
via the Hitchin morphism.
\end{theorem}

Let $X$ be a smooth projective variety.
We denote by $M^r_{X,\text{Dol}}$ the \emph{Dolbeault moduli space}
$M^{P_0,\Omega^1_X\mathrm{-G-ss}}_X$
of semistable Higgs bundles on $X$ with vanishing Chern classes.
As a consequence of Theorem \ref{1stratumimage},
we are able to determine the image of the Hitchin morphism 
from the Dolbeault moduli space, 
when $X$ is a hyperelliptic or an abelian variety.
In characteristic $0$, this coincides with the conjectured image
of the Hitchin morphism in \cite{chenngo}.
The image of the Hitchin morphism for abelian varieties in char $0$ 
was previously determined in \cite[Example 5.1]{chenngo}.

\begin{theorem}[Corollary \ref{hyperelliptic}]\label{hypell}
    Let $X$ be a smooth projective variety that admits an \'etale Galois cover $Y\to X$, where $Y$ is an abelian variety.
    Then the set-theoretic image of $h^r:M^r_{X,\text{Dol}}\to A^{r,\Omega^1_X}_X$ coincides with the subset $(H^0(Y,\Omega^1_Y)^r\sslash S_r)^G$
    of $A^{r,\Omega^1_Y}_Y$, where $G:=\gal(Y/X)$ and $S_r$ denotes the symmetric group in $r$ letters.
    
    In particular, if $X$ is an abelian variety, we have $\mathrm{im}(h^r)=H^0(X,\Omega^1_X)^r\sslash S_r$.
\end{theorem}

\subsection*{Method of proofs}

To study stability of reflexive twisted sheaves under pullback
and prove \Cref{main1},
we first treat the case 
where the cover $\pi:Y\to X$ is finite Galois.
Using the symmetries obtained from the Galois action,
we show that $Y\times_X Y$ is connected in codimension $1$,
that is, for each pair of irreducible components, there is a path
via irreducible components of $Y\times_X Y$ that intersect
in codimension $1$,
if and only if $Y\to X$ does not factor via a non-trivial quasi-\'etale cover.
This is also equivalent to
the induced morphism
$\pi:\pi_1^{\et}(Y_{sm})\to \pi_1^{\et}(X_{sm})$
of \'etale fundamental groups of the smooth loci to be surjective.
We call such morphisms \emph{genuinely ramified in codimension $1$}.

This is an analogue of the equivalence obtained in \cite{bdp} in the Galois case.
We note that our method of proof also recovers the characterization
of genuinely ramified covers in loc. cit. in the Galois case, 
but is much shorter than the original.

We use this characterization of morphisms genuinely ramified in codimension $1$
to show that they preserve slope-stability of reflexive sheaves
via methods similar -- but somewhat more technical -- than in loc. cit.
to obtain \Cref{main1}.

In the proof \Cref{main2}, the Galois action on the stable direct
summands of the reflexive pullback of a stable reflexive twisted sheaf plays
a crucial role. Together with descent results for twisted reflexive sheaves
and Jordan's theorem, we are able to generalize the proof
\cite[Theorem 3.11]{fun} to our setting.

To prove the first part of Theorem \ref{1stratumimage}, 
we define $X_{r,\Hsplit}\to X$ to be an \'etale Galois cover 
that dominates every \'etale Galois cover of degree at most $r$. 
Then it is easy to observe that the inclusion 
$(1)^{\oplus r}_{X_{r,\Hsplit}\to X}\subset(1)^{\oplus r}_{Y\to X}$
holds for all \'etale Galois covers $Y\to X$ dominating $X_{r,\Hsplit}\to X$. 
To prove the reverse inclusion, we show that for any $s\in(1)^{\oplus r}_{Y\to X}$,
the decomposition of $s_{\mid Y}$ as a product of degree $1$ factors descends 
to a decomposition of $s_{\mid X_{r,\Hsplit}}$ as a product of degree $1$ factors.

For the second part, we show that for any $s\in(1)^{\oplus r}_{X_{r,\Hsplit}\to X}$ 
there is an $F$-Higgs bundle $(\mathcal{V},\theta)$ such that $\mathcal{V}$ is
an \'etale trivializable vector bundle of rank $r$ and the characteristic
polynomial of $\theta$ is $s$.
It is straightforward to see that the
characteristic polynomial of any Higgs bundle in $Z^{\FGss,P_0}(1)$ lies in
$(1)^{\oplus r}_{X_{r,\Hsplit}\to X}$.

To prove Theorem \ref{hypell} we first show that when $F$ is an \'etale
trivializable vector bundle, the set-theoretic image of the Hitchin morphism
$h^r:M^{P_0,\FGss}_X\to A^{r,F}_X$ coincides with 
$(1)^{\oplus r}_{X_{r,\Hsplit}\to X}$.
Then the theorem essentially follows from the fact 
that $\Omega^1_X\cong\mathcal{O}_X^{\oplus\dim X}$ for abelian varieties.

\section*{Notation}
We fix an algebraically closed base field $k$ of characteristic $p\geq 0$.
A \emph{variety} is a separated integral scheme of finite type over $k$.
Given a projective variety $X$ we implicitly 
fix an ample line bundle $\mathcal{O}_X(1)$.
For a finite morphism $\pi:Y\to X$ of projective varieties we set $\mathcal{O}_Y(1):=\pi^{\ast}\mathcal{O}_X(1)$.

Let $\pi:Y\to X$ be a morphism between normal projective varieties, 
and let $\mathcal{V}$ be a reflexive sheaf on $X$. 
We denote by $\mathcal{V}_{\mid_Y}$ 
the reflexive pullback $\pi^{[*]}\mathcal{V}:=(\pi^*\mathcal{V})^{\vee\vee}$.
This agrees with the usual pullback if the morphism $\pi$ is flat 
or if $\mathcal{V}$ is a vector bundle.

Let $X$ be a normal projective variety of dimension $d$. 
We denote the category of coherent sheaves on $X$ by $\coh(X)$.
The category $\refl(X)$ of reflexive sheaves on $X$ 
is the full subcategory of $\coh(X)$ where the objects are reflexive sheaves.

Furthermore for $d\ge2$, 
we denote by $\coh(X)_{d,d-1}$ the quotient of $\coh(X)$ by the category of coherent sheaves supported in codimension at least $2$, see \cite[§1.6]{hl}.

An open subset $U\subset X$ is called \emph{big} 
if the complement $X\setminus U$ has codimension at least $2$ in $X$.
 
\section*{Acknowledgements}

AP acknowledges support from the Deutsche Forschungsgemeinschaft (DFG, German Research Foundation) -- Project ID 286237555 (TRR 195) and Project ID 530132094.
We would like to thank Francesco Sala for the question on the image of
the decomposition strata under the Hitchin morphism.
Dario Weißmann would also like to thank Piotr Achinger for inspiring
discussions around fundamental groups.

\section{Reflexive twisted sheaves}
In this section, we introduce our main objects of interest:
(semi)stable reflexive sheaves sheaves together with a twisted endomorphism,
where (semi)stable stands for slope-(semi)stable.
All definitions and lemmas have a direct analogue in the untwisted setting
and the reader familiar with this may skip ahead to
\Cref{lemma:direct image quasi-etale}, where we show that
the direct image of a semistable reflexive sheaf under a quasi-\'etale cover
is semistable, or \Cref{lemma-key},
where we observe some symmetries under the stable
direct summands of the pullback of a stable reflexive sheaf under a Galois cover.

We begin by recalling some standard properties of reflexive sheaves.

\begin{lemma}[{\cite[Proposition 1.8]{reflexive-hartshorne}}]
    Let $\pi:Y\to X$ be a flat morphism between varieties, and let $\mathcal{V}$ be a reflexive sheaf on $X$. Then $\pi^{\ast}\mathcal{V}$
    is a reflexive sheaf on $Y$, i.e., $\pi^{\ast}\mathcal{V}=\mathcal{V}_{\mid Y}$.
\end{lemma}

\begin{lemma}[{\cite[Corollary 1.7]{reflexive-hartshorne}}]
    Let $\pi:Y\to X$ be a proper dominant morphism between normal varieties with
    all fibers of the same dimension, and let $\mathcal{V}$ be a reflexive sheaf on $Y$.
    Then $\pi_*\mathcal{V}$ is a reflexive sheaf on $X$.
\end{lemma}

\begin{definition}[Category of reflexive $F$-twisted sheaves]
\label{fts}
    Let $X$ be a normal variety and let $F$ be a vector bundle on $X$.
    A coherent sheaf $\mathcal{V}$ together with an $\mathcal{O}_X$-linear morphism
    $\theta:\mathcal{V}\to \mathcal{V}\otimes F$ is called an
    \emph{$F$-twisted sheaf}. 
    If moreover $\mathcal{V}$ is reflexive (resp. locally free), it is called a \emph{reflexive $F$-twisted sheaf} (resp. \emph{$F$-twisted bundle}).
    
    Let $(\mathcal{V},\theta_1)$ and $(\mathcal{W},\theta_2)$ be 
    $F$-twisted sheaves.
    A morphism of coherent sheaves $f:\mathcal{V}\to \mathcal{W}$ is a
    \emph{morphism of $F$-twisted sheaves} 
    if $f$ is compatible with $\theta_1$ and $\theta_2$.
    In other words,
    we have $\theta_2\circ f=(f\otimes\id_F)\circ\theta_1$.
     
    We denote the set of morphisms of $F$-twisted sheaves by 
    $\Hom_{F\text{-}TS}(\mathcal{V},\mathcal{W})$ 
    and the category of $F$-twisted sheaves by $F$-$\coh(X)$.
    The full subcategory 
    of $F$-twisted sheaves where the underlying sheaf is
    reflexive is denoted by $\refl_{F-TS}(X)$.
    
    Let $(\mathcal{V},\theta)$ be an $F$-twisted sheaf. 
    A coherent subsheaf $\mathcal{W}\subset \mathcal{V}$ is called an \emph{$F$-twisted subsheaf} 
    if the restricted morphism $\theta|_\mathcal{W}:\mathcal{W}\to \mathcal{V}\otimes F$
    factors through $\mathcal{W}\otimes F$. 
    
    For $d\ge2$ we denote by $F$-$\coh(X)_{d,d-1}$ the quotient of $F$-$\coh(X)$ by the category of $F$-twisted sheaves supported in codimension $\ge2$.
    Setting $F=0$ or $\theta=0$ recovers the categories $\refl(X)$, $\coh(X)$, and $\coh_{d,d-1}(X)$.
\end{definition}

\begin{remark}
    Let $X$ be a normal projective variety of dimension $d\geq 1$. 
    Let $F$ be a vector bundle on $X$.
    Then $F$-$\coh(X)_{d,d-1}$ is an abelian category because it is the quotient of an abelian category.
\end{remark}

\subsection*{Adjunction of push and pull}

Pushforward and reflexive pullback
define functors on the category of twisted sheaves, satisfying the
expected adjunction.

\begin{definition}[Cover]
    A finite separable morphism between normal varieties is called a \emph{cover}. A cover is called \emph{Galois}
    if the induced field extension of function fields is Galois.
\end{definition}

Henceforth, we replace the term 'finite separable morphism' with 'cover'.

\begin{lemma}
    Let $\pi:Y\to X$ be a cover of normal projective varieties, 
    and let $F$ be a vector bundle on $X$. 
    We obtain an adjoint pair of functors 
    \[
        \refl_{F_{\mid Y}-TS}(Y)\underset{\pi_{\ast}}{\overset{\pi^{[\ast]}}
        {\leftrightarrows}}\refl_{F-TS}(X).
    \]
\end{lemma}
 
\begin{proof}
    Recall that Hartogs lemma induces an equivalence
    of categories between reflexive sheaves on a normal variety
    and reflexive sheaves on a big open subvariety via push-pull.

    The direct image $(\pi_{\ast}\mathcal{V},\pi_{\ast}\theta)$ of
    an $F_{\mid Y}$-twisted reflexive sheaf $(\mathcal{V},\theta)$ on $Y$
    becomes an $F$-twisted reflexive sheaf on $X$
    via applying the canonical isomorphism 
    obtained from the projection formula
    $\pi_{\ast}(\mathcal{V}\otimes F_{\mid Y})\cong \pi_{\ast}\mathcal{V} \otimes F$
    to the target of $\pi_{\ast}(\theta)$.

    Then we recover the above adjunction
    via restricting to the locus where both source and target are vector bundles.
    Note that this locus is a big open subvariety by normality.
\end{proof}

\subsection*{Stability for twisted sheaves.}

Let $X$ be a normal projective variety of dimension $d\geq 1$. 
Let $F$ be a vector bundle on $X$.
Recall that the \emph{slope} $\mu(-)$
of a coherent sheaf torsion-free in codimension $1$ is defined as $\deg(-)/\rk(-)$. 
This endows the category $F$-$\coh_{d,d-1}(X)$ with a stability structure
in the sense of \cite{abst}.
This allows us to define (semi)stability
of an $F$-twisted sheaf in $F$-$\coh_{d,d-1}(X)$.
We also obtain the Harder-Narasimhan filtration (HN-filtration),
the destabilizing subobject,
Jordan-Hölder filtrations (JH-filtrations), and the socle.
Taking reflexive hulls of these constructions,
we also obtain all these notions for reflexive $F$-twisted sheaves on $X$.
For reflexive sheaves we spell out these notions in the following:

An $F$-twisted sheaf $\mathcal{V}$ is called
\emph{semi-stable} if $\mathcal{V}$ is torsion-free in codimension $1$ and 
for all non-trivial invariant subsheaves $\mathcal{W}\subseteq \mathcal{V}$
we have $\mu(\mathcal{W})\leq \mu(\mathcal{V})$,
It is called \emph{stable} if the inequality is strict for non-trivial 
invariant subsheaves of smaller rank.
An $F$-twisted sheaf $(\mathcal{V},\theta)$ is called \emph{polystable}
if $\mathcal{V}$ is isomorphic 
(as an $F$-twisted sheaf)
to a direct sum of stable $F$-twisted sheaves of the same slope.

An $F$-twisted subsheaf $\mathcal{W}\subset \mathcal{V}$
of a torsion-free $F$-twisted
sheaf $\mathcal{V}$ is called \emph{saturated}
if the quotient $\mathcal{V}/\mathcal{W}$ is torsion-free.

Let $\mathcal{V}$ be a reflexive $F$-twisted sheaf.
The \emph{HN-filtration} of $\mathcal{V}$ is the unique filtration
by saturated $F$-twisted subsheaves
\[
0\subset\mathcal{V}_1\subset\dots\subset\mathcal{V}_m=\mathcal{V}
\]
such that each quotient $\mathcal{V}^i:=\mathcal{V}_i/\mathcal{V}_{i-1}$ 
is semistable and 
\[
\mu(\mathcal{V}^1)>\mu(\mathcal{V}^2)>\dots>\mu(\mathcal{V}^m).
\]
The subsheaf $\mathcal{V}_1$ is called the \emph{maximal destabilizing subsheaf} 
of $\mathcal{V}$, and, as the name suggests, 
is the subsheaf whose slope is maximal among all subsheaves of $\mathcal{V}$ 
(see \cite[Proposition 1.9]{abst}). 

Let $\mathcal{V}$ be a semistable reflexive $F$-twisted sheaf.
A \emph{JH-filtration} of $\mathcal{V}$
is a filtration of $\mathcal{V}$ by saturated $F$-twisted subsheaves
\[
    0 \subset \mathcal{V}_1 \subset \dots \subset \mathcal{V}_m=\mathcal{V}
\]
such that the successive quotients $\mathcal{V}_i/\mathcal{V}_{i-1}$
are stable of slope $\mu(\mathcal{V})$.
In contrast to the HN-filtration there may be several JH-filtrations.
However, the reflexive hull of the associated graded 
$\bigoplus_{i=1}^m \mathcal{V}_i/\mathcal{V}_{i-1}$ is unique.
The unique maximal polystable saturated subsheaf of a semistable sheaf
$\mathcal{V}$ of slope $\mu(\mathcal{V})$
is called the \emph{socle} of $\mathcal{V}$. 

The standard argument in \cite[Cor 1.2.8]{hl} carries over to show 
that stable reflexive twisted shaves are simple:
\begin{lemma}
\label{lemma-stable-simple}
    Let $X$ be a normal projective variety
    and $F$ a vector bundle on $X$. 
    Let $\mathcal{V}$ and $\mathcal{W}$ be stable
    reflexive $F$-twisted sheaves such that $\mu(\mathcal{V})=\mu(\mathcal{W})$. 
    Then any non-zero morphism $\mathcal{V}\to\mathcal{W}$ of $F$-twisted sheaves
    is an isomorphism and in this case we have
    \[
    \Hom_{F-TS}(\mathcal{V},\mathcal{W})\cong k.
    \]
\end{lemma}

\begin{proof}
As a direct consequence of the
definition of stability a non-zero morphism of stable
$F$-twisted shaves of the same slope is injective and has cokernel
supported in codimension $2$ or higher. If both sheaves are in addition
reflexive sheaves on a normal projective variety, 
we find that the morphism is an isomorphism.
    
Note that $F$-twisted homomorphisms form a $k$-subvector space 
of the $k$-vector space of homomorphisms.
As $X$ is projective (proper suffices), we find that 
$F$-twisted homomorphisms
of $F$-twisted sheaves are a finite dimensional vector space.
Endomorphisms also form a (possibly non-commutative) $k$-algebra.
For a stable reflexive $F$-twisted sheaf $\mathcal{V}$ every non-zero morphism is invertible, thus $\mathrm{End}_{F-TS}(\mathcal{V})$ forms a finite dimensional skew field over $k$. As $k$ is algebraically closed, we conclude $\mathrm{End}_{F-TS}(\mathcal{V})=k$.
\end{proof} 

Recall that we denote by $(-)_{\mid Y}$ the reflexive 
pullback of an object along a morphism $Y\to X$.
We recall the behavior of the slope under (reflexive) pushforward and pullback:

\begin{lemma}[{\cite[Lemma 3.2.1]{hl}}]
\label{lemma-push-pull-slope}
Let $\pi:Y\to X$ be a cover of normal projective varieties.
Consider a reflexive sheaf $\mathcal{V}$ on $Y$ and a reflexive sheaf $\mathcal{W}$ on $X$. Then we have the following:
    \begin{enumerate}[(i)]
		\item $\mu(\mathcal{V})=d(\mu(\pi_{\ast}\mathcal{V})-\mu(\pi_{\ast}\mathcal{O}_Y))$.
		\item $\mu(\mathcal{W}_{\mid Y})=d\mu(\mathcal{W})$.
    \end{enumerate}  
\end{lemma}

\begin{lemma}[Analogue of Lemma 2.4, \cite{fun}, in the twisted setting]
\label{lemma-twisted-descend}
    Let $Y\to X$ be a Galois cover of normal varieties with Galois group $G$.
    Let $F$ be a vector bundle on $X$. 
    Let $\mathcal{V}$ be an $F$-twisted reflexive sheaf on $X$. 
    Then an $F_{\mid Y}$-twisted saturated 
    subsheaf $\mathcal{W}\subset \mathcal{V}_{\mid Y}$
    descends to an $F$-twisted subsheaf of $\mathcal{V}$
    if and only if $\mathcal{W}\subset \mathcal{V}_{\mid Y}$
    is $G$-invariant, that is, compatible with the $G$-linearization of
    $\mathcal{V}_{\mid Y}$.
\end{lemma}

\begin{proof}
    We proceed as in \cite[Lemma 2.4]{fun}. Let $\eta_X$ and $\eta_Y$ denote the generic points of $X$ and $Y$ respectively. 

    One implication is immediate. 
    For, the other let $\mathcal{W}\subset\mathcal{V}_{\mid Y}$
    be a saturated $G$-invariant $F_{\mid Y}$-twisted subsheaf.
    Restricting the inclusion to the generic point $\eta_Y$,
    we obtain a $G$-invariant $(F_{\eta_X})_{\mid \eta_Y}$-twisted 
    subvector space 
    $\mathcal{W}_{\eta_Y}\subset(\mathcal{V}_{\eta_X})_{\mid \eta_Y}$.

    The field extension $\kappa(Y)/\kappa(X)$ is a $G$-torsor.
    We can thus apply descent theory to obtain a $F_{\eta_X}$-twisted subspace 
    $\mathcal{W}'_{\eta_X}\subset \mathcal{V}_{\eta_X}$ such that $\mathcal{W}'_{\eta_X}\otimes_{\kappa(X)}\kappa(Y)\cong\mathcal{W}_{\eta_Y}$ 
    as an $(F_{\eta_X})_{\mid \eta_Y}$-twisted subvector space of $(\mathcal{V}_{\eta_X})_{\mid \eta_Y}$.

    By \cite[Proposition 1]{langton}, there is a unique saturated
    subsheaf $\mathcal{W}'\subset\mathcal{V}$ that induces the inclusion 
    $\mathcal{W}'_{\eta_X}\subset \mathcal{V}_{\eta_X}$.
    Note that loc. cit. is a functorial construction.
    Thus, $\mathcal{W}'$ is a $F$-twisted subsheaf of $\mathcal{V}$.

    Consider the reflexive $F_{\mid Y}$-twisted subsheaf
    $\mathcal{W}'_{\mid Y}$ of $\mathcal{V}_{\mid Y}$.
    Note that $\mathcal{W}'_{\mid Y}\subset \mathcal{V}_{\mid Y}$
    and $(\mathcal{W}'\otimes F)_{\mid Y}\subset (\mathcal{V}\otimes F)_{\mid Y}$
    are saturated.
    At the generic point $\eta_{Y}$ this coincides
    with the $F_{\eta_Y}$-twisted subspace
    $\mathcal{W}_{{\eta}_Y}$.
    Again, by the uniqueness of the saturated extension
    in \cite[Proposition 1]{langton} we conclude
    that $\mathcal{W}'_{\mid Y}\cong \mathcal{W}$ as $F_{\mid Y}$-twisted
    subsheaves of $\mathcal{V}_{\mid Y}$.
\end{proof}

\begin{lemma}
    \label{lemma-pullback-stable}
    Let $\pi:Y\to X$ be a cover
    of normal projective varieties. 
    Let $F$ be a vector bundle on $X$ 
    and let $\mathcal{V}$ be a reflexive $F$-twisted sheaf. 
    Then we have the following:
    \begin{enumerate}[(i)]
        \item $\mathcal{V}_{\mid Y}$ is semistable if and only 
        if $\mathcal{V}$ is semistable.
        \item If $\pi$ is Galois and 
        $\mathcal{V}$ is polystable, then so is 
        $\mathcal{V}_{\mid Y}$.
        \item If $\pi$ is Galois prime to $p$, 
            then $\mathcal{V}$ is polystable if and only if $\mathcal{V}_{\mid Y}$ is polystable.
    \end{enumerate}
\end{lemma}

\begin{proof}
    Note that it suffices to show (i) after replacing $Y\to X$ by its Galois closure.
    Then the claim follows from the uniqueness of the destabilizing $F$-twisted 
    subsheaf $\mathcal{D}\subseteq\mathcal{V}_{\mid Y}$. 
    As $\mathcal{D}$ is unique it inherits the linearization of
    $\mathcal{V}_{\mid Y}$ as an $F$-twisted sheaf and 
    thus descends as a saturated $F_{\mid Y}$-twisted subsheaf 
    $\mathcal{D'}\subset \mathcal{V}$ in codimension $1$ by \Cref{lemma-twisted-descend}.
    As the slope only depends on the isomorphism
    class on a big open and the behaviour of the slope under pullback, we
    find $\mu(\mathcal{D'})>\mu(\mathcal{V})$ if and only if
    $\mu(\mathcal{D})>\mathcal{V}_{\mid Y}$.

    (ii) follows via an analogous argument after observing that it suffices
    to show the claim for stable $\mathcal{V}$ 
    and replacing the maximal destabilizing subsheaf by the socle.

    (iii) If $\pi$ is prime to $p$, then the trace splits $\mathcal{O}_X\to\pi_{\ast}\mathcal{O}_Y$ and we 
    can argue as in \cite[Lemma 3.2.3]{hl}.
\end{proof}

Recall that an open subset whose complement has codimension $\ge2$ is called \emph{big}.
We recall the notion of a quasi-\'etale cover.

\begin{definition}[Quasi-\'etale]
\label{def: quasi-etale}
    Let $Y\to X$ be a cover of a normal projective variety $X$.
    We call $Y\to X$ \emph{quasi-\'etale} if it is \'etale when restricted
    to a big open subvariety of $X$.
\end{definition}

\begin{lemma}
    \label{lemma:direct image quasi-etale}
    Let $\pi:Y\to X$ be an quasi-\'etale cover of a normal projective variety $X$.
    Then $\pi_{\ast}\mathcal{O}_Y$ has slope $0$.
    
    Let $F$ be a vector bundle on $X$ and $\mathcal{V}$ a semistable 
    $F_{\mid Y}$-twisted reflexive sheaf on $Y$.
    Then $\pi_{\ast}\mathcal{V}$ is a semistable $F$-twisted reflexive sheaf of 
    slope $\mu(\mathcal{V})/\deg(\pi)$.
\end{lemma}

\begin{proof}
    Note that we can show slope $0=\mu (\pi_{\ast}\mathcal{O}_Y)$
    after pullback by a Galois cover.
    Let $\pi':Y'\xrightarrow{\varphi} Y\to X$ be the Galois closure of $\pi$. 
    Then $Y'\to X$ is still quasi-\'etale, 
    say both $\pi$ and $\pi'$ are \'etale when
    restricted to the big open $U\subset X$.
    By construction 
    $\pi'^{-1}(U)\times_U \pi^{-1}(U)$ is a disjoint union of $\deg(\pi)$
    copies of $\pi'^{-1}(U)$
    and we obtain $(\pi_{\ast}\mathcal{O}_Y)_{\mid \pi'^{-1}(U)}
    \cong \mathcal{O}_{\pi'^{-1}(U)}^{\oplus \deg(\pi)}$.
    Further, $\pi'^{-1}(U)\subset Y$ is a big open and we conclude that
    $(\pi_{\ast}\mathcal{O}_Y)_{\mid Y'}\cong \mathcal{O}_{Y'}^{\oplus \deg(\pi)}$
    which has slope $0$.
    
    By \Cref{lemma-pullback-stable} it suffices to check semistability
    of $\pi_{\ast}\mathcal{V}$ after pullback to $Y'$.
    Again, since $\pi'^{-1}(U)\times_U \pi^{-1}(U)$ 
    is a disjoint union of $\deg(\pi)$ copies of $\pi'^{-1}(U)$ 
    we obtain an isomorphism 
    $(\pi_{\ast}\mathcal{V})_{\mid \pi'^{-1}(U)}
    \cong\mathcal{V}_{\mid \pi'^{-1}(U)}^{\oplus \deg(\pi)}$.
    Moreover, $\pi'^{-1}(U)\subset Y$ is a big open and we conclude that
    $(\pi_{\ast}\mathcal{V})_{\mid Y'}\cong 
    \mathcal{V}^{\oplus \deg(\pi)}_{\mid Y'}$
    is a direct sum of semistable $F_{\mid Y'}$-twisted reflexive sheaves
    of the same slope and is thus semistable. 

    The claim regarding the slope of $\pi_{\ast}\mathcal{V}$ now follows
    immediately from \Cref{lemma-push-pull-slope}.
\end{proof}

Note that a cover need not be flat by definition.
Nonetheless, descent along a Galois cover of a saturated subsheaf
that descends can be phrased in terms of a linearization as in the flat case:

\begin{lemma}
\label{lemma-key}
    Let $\pi:Y\to X$ be a Galois cover of normal projective varieties with Galois group $G$. 
    Let $F$ be a vector bundle on $X$ and let $\mathcal{V}$ be a stable
    $F$-twisted reflexive sheaf. 
    Then $\mathcal{V}_{\mid Y}\cong \bigoplus_{i=1}^n \mathcal{W}_i^{\oplus e}$ 
    for some pairwise non-isomorphic stable $F_{\mid Y}$-twisted 
    reflexive sheaves $\mathcal{W}_i$ on $Y$ such that $G$ acts transitively on the isomorphism classes of the $\mathcal{W}_i$.
\end{lemma}

\begin{proof}
    By Lemma \ref{lemma-pullback-stable} (ii), $\mathcal{V}_{\mid Y}$ is polystable.
    Consider a stable direct summand $\iota:\mathcal{W}_1\subseteq\mathcal{V}_{\mid Y}$.
    Denote the image of the orbit of $\iota$ under the Galois action by $\mathcal{W}'$,
    i.e., $\mathcal{W}'$ is the image of 
    \[
        \psi: \bigoplus_{\sigma\in G}\sigma^{\ast} \mathcal{W}_1\xrightarrow{\oplus \varphi_{\sigma}\circ \sigma^{\ast}\iota}\mathcal{V}_{\mid Y},
    \]
    where $\varphi_{\sigma}:\sigma^{\ast}\mathcal{V}_{\mid Y}\to\mathcal{V}_{\mid Y}$
    is the $G$-linearization induced by $\mathcal{V}$.
    Note that $\varphi_{\sigma}$ is compatible with the $F_{\mid Y}$-twisted structure on $\mathcal{V}_{\mid Y}$. Furthermore, as the direct summand 
    $\iota: \mathcal{W}_1\subseteq\mathcal{V}_{\mid Y}$
    is an $F_{\mid Y}$-twisted subsheaf, then so is $\mathcal{W}'\subseteq\mathcal{V}_{\mid Y}$.
    By construction $\mathcal{W}'\subseteq \mathcal{V}_{\mid Y}$ inherits the $G$-linearization of $\mathcal{V}$ making it an $F_{\mid Y}$-twisted subsheaf. As $\mathcal{W}'$ is the image of a morphism
    of semistable $F_{\mid Y}$-twisted sheaves of the same slope, it is semistable of the same slope as well.
    Thus, there exists an $F$-twisted subsheaf $\mathcal{W}\subseteq \mathcal{V}$
    pulling back to $\mathcal{W}'\subseteq\mathcal{V}_{\mid Y}$. As $\mathcal{W}$ has the same slope as $\mathcal{V}$, we conclude that $\mathcal{W}=\mathcal{V}$, i.e., $\psi$ is surjective.

    As the associated graded object of the JH-filtration is unique, the surjectivity of $\psi$ implies that every stable direct summand of the polystable
    $F_{\mid Y}$-twisted sheaf is of the form $\sigma^{\ast}\mathcal{W}_1$
    for some $\sigma\in G$. Furthermore, from the transitive action of $G$
    we conclude that all direct summands $\mathcal{W}_i$ appear with the same multiplicity.
\end{proof}

\section{Genuinely ramified morphisms and stability}

The goal of this section is to prove several generalizations of 
\cite[Theorem 1.2]{bdp}:
pullback by a genuinely ramified cover $Y\to X$ 
preserves stability of twisted bundles and
reflexive pullback by a genuinely ramified cover with 
smooth target preserves stability of reflexive sheaves. 
We also study stability under reflexive pullback
in the normal setting and obtain a similar criterion 
for when such a pullback preserves stability of
reflexive sheaves.

Recall that a cover $Y\to X$ of normal projective varieties
is genuinely ramified if and only if one of the following 
is satisfied, see \cite[Theorem 1.1]{bdp}:
\begin{itemize}
    \item the cover $Y\to X$ does not factor via a non-trivial \'etale
    cover $X'\to X$.
    \item the induced morphism 
    $\pi_1^{\et}(Y)\to\pi_1^{\et}(X)$ is surjective.
    \item the fibre product $Y\times_X Y$ is connected.
\end{itemize}

Our first main result is an analogous characterization
of covers that do not factor via a non-trivial quasi-\'etale cover.
We call such covers \emph{genuinely ramified in codimension 1}
(\Cref{def: genuinely ramified in codim}).

\begin{theorem}
\label{thm:bdp analogue surjective on fundamental group}
    Let $\pi:Y\to X$ be a Galois cover of normal varieties.
    Then the following are equivalent:
    \begin{enumerate}[(i)]
        \item The cover $Y\to X$ is genuinely ramified in codimension $1$.
        \item The self-intersection $Y\times_X Y$ is connected in codimension $1$,
        see \Cref{definition-connected-in-codim}.
        \item The induced morphism $\pi_\ast:\pi_1^{\et}(Y_{sm})\to \pi_1^{\et}(X_{sm})$ is surjective, where $(-)_{sm}$ denotes the smooth locus,
        see \Cref{lemma:induced morphism etale fundamental group smooth loci}
        for the construction of the morphism.
    \end{enumerate}
\end{theorem}

Our techniques rely on the symmetries
obtained from the action of the Galois group of $Y/X$ on $Y\times_X Y$
and thus differ from the approach in \cite{bp}.

Using this characterization, we obtain that covers genuinely ramified 
in codimension $1$ preserve stability of reflexive twisted sheaves under pullback:

\begin{theorem}
\label{thm:bdp analogue stable}
    Let $\pi:Y\to X$ be a genuinely ramified morphism of normal projective varieties.
    Let $F$ be a vector bundle on $X$. Then an $F$-twisted bundle $(\mathcal{V},\theta)$ on $X$ is stable 
    iff $(\mathcal{V}_{\mid Y},\theta_{\mid Y})$ is stable
    as an $F_{\mid Y}$-twisted bundle.
    
    Furthermore, if $Y\to X$ is genuinely ramified in codimension $1$, 
    then an $F$-twisted reflexive sheaf $(\mathcal{V},\theta)$ on $X$ 
    is stable if and only if $(\mathcal{V}_{\mid Y},\theta_{\mid Y})$ 
    is stable as an $F_{\mid Y}$-twisted reflexive sheaf.
\end{theorem}

\subsection{\'Etale fundamental group of the smooth locus}

Let $X$ be a normal variety. In this subsection,
we study the \'etale fundamental group of the smooth locus of $X$.
Although all the results should be well-known, 
we could not find them in the literature and include the short arguments.

We provide an interpretation of the \'etale fundamental
group $\pi_1^{\et}(X_{sm})$ of the smooth locus $X_{sm}$ of $X$
via quasi-\'etale covers of $X$.
Given a cover $Y\to X$ we obtain an induced morphism
$\pi_1^{\et}(Y_{sm})\to \pi_1^{\et}(X_{sm})$.
Furthermore, if $X$ is also projective, then
$\pi_1^{\et}(X_{sm})$ is topologically finitely generated.
In particular, there are only finitely many quasi-\'etale covers
of $X$ of fixed degree if $X$ is a normal projective variety.

\begin{lemma}
    \label{lemma:quasi-etale smooth locus}
    Let $X$ be a normal variety.
    Then quasi-\'etale covers of $X$ are in a 1:1 correspondence to
    \'etale covers of the smooth locus $X_{sm}$ of $X$.
\end{lemma}

\begin{proof}
    Given an \'etale cover $Y\to X_{sm}$ we obtain
    a quasi-\'etale cover
    $Y'\to X$, where $Y'$ denotes the normal closure of $X$
    in the function field of $Y$.
    
    Conversely, given a quasi-\'etale cover $Y'\to X$,
    we obtain a quasi-\'etale cover of $Y\to X_{sm}$ by restriction.
    By purity of branch locus, 
    \cite[\href{https://stacks.math.columbia.edu/tag/0BMB}{Tag 0BMB}]{sp},
    we find that $Y\to X_{sm}$ is an \'etale cover.

    It is straightforward to check that these constructions are inverse to 
    each other.
\end{proof}

\begin{lemma}
    Let $X$ be a normal projective variety.
    Then the \'etale fundamental group of the smooth
    locus $X_{sm}$ of $X$ is topologically finitely generated.
\end{lemma}

\begin{proof}
    We interpret \'etale covers of $X_{sm}$ as quasi-\'etale covers of $X$
    via \Cref{lemma:quasi-etale smooth locus}.
    
    Note that the general complete intersection curve of large enough degree
    in $X$ also lies in the smooth locus
    as the smooth locus is a big open in $X$.
    Indeed, for $N\gg 0$, the general hyperplane $H$ induced by a section of
    $H^0(X,\mathcal{O}_X(N))$ intersects the locus $X\setminus X_{sm}$ transversally.
    Furthermore, the general such $H$ is also a normal projective variety by 
    the normal version of Bertini's theorem, see \cite[Theorem 7 and 7']{sei}.

    Let $C\hookrightarrow X$ be such a curve.
    Let $Y\to X$ be a quasi-\'etale cover.
    Note that by \cite[Corollaire 6.11 (3)]{jou} we have
    that $C\times_X Y$ is irreducible.
    Thus, we obtain the surjectivity of 
    $\pi_1^{\et}(C)\to \pi_1^{\et}(X_{sm})$
    by \cite[\href{https://stacks.math.columbia.edu/tag/0BN6}{Tag 0BN6}]{sp},
    and conclude by the well-known fact that $\pi_1^{\et}(C)$ is topologically
    finitely generated.
\end{proof}

As a topologically finite generated group has only finitely many homomorphisms
to a finite group we obtain:

\begin{cor}
    \label{cor: quasi-etale of fixed degree}
    Let $X$ be a normal projective variety.
    Then $X$ has only finitely many quasi-\'etale covers of fixed degree.
\end{cor}

\begin{lemma}
    \label{lemma:induced morphism etale fundamental group smooth loci}
    Let $\pi:Y\to X$ be a cover of a normal variety $X$.
    Let $y$ be a closed point in the smooth locus of $Y$
    such that $x:=\pi(y)$ lies in the smooth locus of $X$.
    Then we obtain an induced morphism 
    $\pi_{\ast}:\pi_1^{\et}(Y_{sm},y)\to \pi_1^{\et}(X_{sm},x)$.   
\end{lemma}

\begin{proof}
    We interpret an \'etale cover of the smooth locus
    with quasi-\'etale cover via \Cref{lemma:quasi-etale smooth locus}. 
    
    Let $X'\to X$ be a quasi-\'etale cover.
    Taking the normalization of the fibre product $(Y\times_X X')_{red}$
    yields a disjoint union
    of quasi-\'etale covers on $Y$ which we interpret as
    a disjoint union of \'etale covers of $Y_{sm}$.
    This assignment is functorial
    for morphisms of \'etale covers of $X_{sm}$
    and commutes with the fiber functors at $x$ resp. $y$.
    We obtain the desired morphism of \'etale fundamental groups.
\end{proof}

\begin{remark}
    As a cover $Y\to X$ is generically \'etale,
    there exist $k$-points in the smooth locus of $Y$
    mapping to the smooth locus  of $X$.
    Further, the \'etale fundamental group
    does not depend on the choice of base point and
    we denote the morphism constructed
    by $\pi_1^{\et}(Y_{sm})\to \pi_1^{\et}(X_{sm})$
    suppressing this choice.
\end{remark}

\subsection{Covers genuinely ramified in codimension 1}

\ 

Recall that for a Galois cover $Y\to X$ of normal projective varieties the
automorphism group of $Y/X$ corresponds
to the Galois group $G$ of the extension of function fields $\kappa(Y)/\kappa(X)$.
This action is obtained via thinking of $Y$ as the normal closure of $X$ in the function
field of $Y$.
Furthermore, the irreducible components of $Y\times_X Y$
correspond to elements of $G$ via 
$G\times Y\to Y\times_X Y, (\sigma,y)\mapsto (y,\sigma y)$. 

In this subsection we use the symmetries
obtained from the Galois action
on the components of $Y\times_X Y$ to study the connectedness of $Y\times_X Y$
and relate it to the ramification of $Y\to X$.

\begin{definition}[Genuinely ramified in codim at most $i$]
    \label{def: genuinely ramified in codim}
    Let $Y\to X$ be a cover of normal varieties of dimension $d\geq 1$.
    Let $1\leq i \leq \dim(X)$ be an integer.
    
    We say that $Y\to X$ is \emph{genuinely ramified in codimension at most $i$}
    if for every non-trivial
    intermediate cover $Y\to Y'\to X$ of normal varieties
    the cover $Y'\to X$ is ramified at a point of codimension at most $i$.

    In codimension $1$ we abbreviate \emph{genuinely ramified in codimension at most $1$}
    to \emph{genuinely ramified in codimension $1$} as that is the
    smallest codimension the ramification can have.

\end{definition}

\begin{remark}
    A cover $Y\to X$ genuinely ramified in codimension at most $i$ is genuinely
    ramified as the name suggests by definition.
\end{remark}

\begin{lemma}
    Let $Y\to X$ be a Galois cover of normal varieties.
    Fix an integer $1\leq i\leq \dim(X)$.
    Then $Y\to X$ is genuinely ramified in codimension at most $i$
    if and only if for all intermediate Galois covers $X'\to X$
    of $Y\to X$ we have that $X'\to X$ is ramified in codimension at most $i$.
\end{lemma}

\begin{proof}
    This is immediate as the locus where a cover $X'\to X$ is \'etale
    agrees with the locus where its Galois closure is \'etale.
\end{proof}

\begin{definition}[Connected in codimension at most $i$; connecting group]
    \label{definition-connected-in-codim}
    Let $Y\to X$ be a Galois cover of normal varieties
    of dimension at least $1$
    with Galois group $G$.
    We have a surjective morphism 
    \[
        G\times_k Y\to Y\times_X Y, (\sigma,y)\mapsto (y,\sigma y)
    \]
    which induces a one-to-one correspondence of irreducible components.
    We denote the irreducible component corresponding to $\sigma\in G$ by $Y_{\sigma}$.

    For $\sigma,\tau\in G$ we denote by $Y_{\sigma,\tau}$
    the fibre product 
    \[
    \begin{tikzcd}
        Y_{\sigma,\tau}\ar[r] \ar[d]& Y \ar{d}{(\sigma,\tau)}\\
        Y \ar{r}{\diag_{Y/X}} & Y\times_X Y.
    \end{tikzcd}
    \]
    As the diagonal of $Y/X$ is a closed immersion,
    $Y_{\sigma,\tau}$ is a closed subscheme of $Y$ and we consider its codimension in $Y$.
     
    We say that $\sigma$ and $\tau$ (or $Y_{\sigma}$ and $Y_{\tau}$)
    are \emph{connected in codimension at most $i$}
    if there exists a sequence of elements $\sigma=\sigma_1,\sigma_2,\dots,\sigma_n=\tau$
    such that $Y_{\sigma_j,\sigma_{j+1}}$ has codimension $\leq i$ in $Y$ 
    for $j=1,\dots, n-1$.
    
    We say that $Y\times_X Y$ is \emph{connected in codimension at most $i$} 
    if every irreducible component $Y_{\sigma}$, with $\sigma\in G$, 
    is connected to $Y_{e_G}$ in codimension at most $i$.

    We denote the subgroup of $G$ generated by the elements $\sigma\in G$
    which are  connected to $e_G$ in codimension at most $i$ by $G_i$.
    We call $G_i$ the \emph{$i$-th connecting group} of $Y\to X$.
\end{definition}

\begin{remark}
    In the setting of Definition \ref{definition-connected-in-codim}
    we observe for $\sigma,\tau,\rho\in G$:
    \begin{itemize}
        \item Points $y$ in $Y_{\sigma,\tau}$
        correspond to points $y$ in $Y$ such that $\sigma y = \tau y$
        and the induced isomorphism 
        $\sigma^{-1}\tau:\kappa(y)\to \kappa(y)$ is trivial.
        \item $Y_{\sigma\rho,\tau\rho}\cong Y_{\sigma,\tau}$
        \item $Y_{\sigma,\tau}\cong Y_{\rho\sigma,\rho\tau}$.
        \item $Y_{e,\tau}\cong Y_{e,\rho\tau\rho^{-1}}$.
        \item $Y_{\sigma,e}\cong Y_{\rho\sigma\rho^{-1},e}$.
    \end{itemize}
\end{remark}

The connectedness of $Y\times_X Y$ is measured by an ascending sequence
of normal subgroups of the Galois group of the Galois cover $Y\to X$:

\begin{lemma}
\label{lemma-connecting-group}
    Let $Y\to X$ be a Galois morphism of normal varieties of dimension $d\geq 1$
    with Galois group $G$.
    Then the connecting groups form a normal series of normal subgroups of $G$
    \[
         G_1\subseteq \dots \subseteq G_d\subseteq G.
    \]
    Furthermore, the following are equivalent: 
        \begin{enumerate}[(i)]
            \item $Y\to X$ is genuinely ramified in codimension at most $i$,
            \item $Y\times_X Y$ is connected in codimension at most $i$, and
            \item $G_i=G$.
        \end{enumerate}
\end{lemma}
\begin{proof}
    To show that $G_i$ is a normal subgroup of $G$,
    we show that for $\sigma\in G_i$ and $\tau\in G$ t
    he conjugate $\tau\sigma\tau^{-1}$ also lies in $G_i$.
    It suffices to show that conjugation does not change 
    the codimension of $Y_{\sigma,\nu}$ in $Y$, i.e., 
    \[
        \mathrm{codim}_Y Y_{\sigma,\nu}=\mathrm{codim}_Y Y_{\tau\sigma\tau^{-1},\tau\nu\tau^{-1}},
    \]    
    where $\sigma,\nu,\tau\in G$.
    By considering the morphisms
    \[
        Y\xrightarrow{\tau^{-1}}Y\xrightarrow{(\sigma,\nu)}Y\times_X Y\xrightarrow{(\tau,\tau)}Y\times_X Y
    \]
    we obtain isomorphisms
    \[
    \begin{tikzcd}
        Y_{\sigma,\nu} \ar[d,"\tau","\sim"'] \ar[r,hook]& Y \ar[d,"\tau","\sim"']\\
        Y_{\tau\sigma\tau^{-1},\tau\nu\tau^{-1}} \ar[r,hook] & Y.
    \end{tikzcd}
    \]
    Thus they have the same codimension in $Y$.

    $(ii) \Leftrightarrow (iii):$ The equivalence of $G_i=G$ and $Y\times_X Y$ being connected in codimension $\leq i$ is a reformulation of the definitions.
    
    $(i)\Rightarrow (iii):$ Suppose that the morphism $Y\to X$ is genuinely ramified in codimension $\le i$. 
    Consider the factorization $Y\to Y':=Y/G_i\to X$.  
    Since $G_i$ is normal in $G$, the map $Y'\to X$ is Galois with Galois group $G'=G/G_i$. 
    If $Y'\neq X$, then the map $Y'\to X$ is ramified at a point $y'$ of codimension $\le i$ by assumption. 
    As $X=Y'/G'$, we obtain that there exists an element $\overline{\tau}\in G' \setminus \{e_{G'}\}$
    such that $\overline{\{y'\}}\subseteq Y'_{e_{G'},\overline{\tau}}$
    by \cite[\href{https://stacks.math.columbia.edu/tag/0BTF}{Tag 0BTF, (9)}]{sp}.
    The quotient morphism $\pi:Y\to Y'$ is surjective and so is the morphism
    \[
        \pi^{-1}(Y'_{e_{G'},\overline{\tau}})=\bigcup_{\sigma\in G_i,\tau\in \overline{\tau}} 
        Y_{\sigma,\tau}\to Y'_{e_{G'},\overline{\tau}},
    \]
    where the union of the $Y_{\sigma,\tau}$'s is taken set-theoretically.
    As $\pi$ is finite,
    we find elements $\tau, \sigma\in G$
    such that $Y_{\sigma,\tau}$ has codimension at most $i$ in $Y$. 
    As $Y_{\sigma}$ is connected to $Y_{e_G}$ in codimension at most $i$,
    the same is true for $Y_{\tau}$.
    Thus, $\tau$ lies in $G_i$, which contradicts $\overline{\tau}\neq e_{G/G_i}$. It follows that $G_i=G$.
    
    $(iii)\Rightarrow (i):$ Conversely, 
    suppose that $G_i=G$, and let $Y\to Y'\to X$ be any intermediate Galois cover. 
    Note that $G_i=G$ implies that $Y\times_X Y$ is connected.
    In particular, $Y'\to X$ is not \'etale.
    
    Assume by way of contradiction that there is a point 
    $y'\in Y'$ of codimension greater $i$ at which $Y'\to X$ is ramified. 
    Then the preimage of $\overline{\{y'\}}$ in $Y$
    also has codimension $>i$ in $Y$ as $Y\to Y'$ is finite. 
    Moreover, it is contained in a union of $Y_{\sigma,\tau}$ 
    for some $\sigma,\tau\in G=G_i$. 
    However, for any pair of elements $\sigma,\tau$ of $G_i$, $Y_{\sigma,\tau}$ 
    has codimension at most $i$ in $Y$, which is a contradiction. 
    We conclude that $Y\to X$ is genuinely ramified in codimension at most $i$.
\end{proof}

\begin{remark}
    Let $Y\to X$ be a Galois morphism of normal projective varieties 
    of dimension $d\geq 1$ with Galois group $G$.
    Then we recover the equivalence of the defining conditions
    of $Y\to X$ being genuinely ramified from Lemma \ref{lemma-connecting-group}.
    The proof given only works for Galois morphisms but has the advantage
    of being more concise than the original one given in \cite{bdp}.
    It also removes the projectivity assumption on $X$.
\end{remark}

\begin{lemma}
    Let $\pi:Y\to X$ be a cover of a normal variety $X$.
    Then $\pi$ is genuinely ramified in codimension $1$ if 
    and only if 
    $\pi_\ast:\pi_1^{\et}(Y_{sm})\to \pi_1^{\et}(X_{sm})$ is surjective.
\end{lemma}

\begin{proof}
    If the induced morphism
    $\pi_\ast:\pi_1^{\et}(Y_{sm})\to \pi_1^{\et}(X_{sm})$
    is not surjective, then consider its image $H$.
    As $H$ is a proper closed subgroup of the profinite group $\pi_1^{\et}(X_{sm})$, it is contained in
    a proper open subgroup $H'$ of finite index, e.g., by \cite[Proposition 2.1.4]{RZ}.
    
    The non-trivial quotient $\pi_1^{\et}(X_{sm})/H'$
    corresponds to a non-trivial quasi-\'etale cover $X'\to X$.
    By construction, the normalization of $(X'\times_X Y)_{red}$ 
    is a disjoint union of copies of $Y$.
    In particular, $Y\to X'\times_X Y \to X' \to X$
    is a factorization of $Y\to X$ via the non-trivial
    quasi-\'etale cover $X'\to X$. Thus, $Y\to X$ is not
    genuinely ramified in codimension $1$.

    For the converse,
    recall that $\pi_1^{\et}(Y_{sm})\to \pi_1^{\et}(X_{sm})$
    is surjective if and only if every \'etale cover
    of $X_{sm}$ maps to an \'etale cover of $Y_{sm}$
    by general properties of a morphism of Galois categories,
    \cite[\href{https://stacks.math.columbia.edu/tag/0BN6}{Tag 0BN6}]{sp}
    (recall that "cover" includes irreducible).

    Now, if $Y\to X$ is not genuinely ramified in codimension $1$,
    then there exists a factorization of $Y\to X$
    via a non-trivial quasi-\'etale cover $\varphi:X'\to X$.
    Consider the Galois closure $X''\to X$ of $X'\to X$.
    Then the normalization of $(X''\times_X X')_{red}$ is a disjoint union
    of $\deg(\varphi)$ copies of $X'$ and in particular not connected.
    Then neither is the normalization of 
    $(X''\times_X Y)_{red} = ((X''\times_X X') \times_{X'} Y)_{red}$.
\end{proof}

As for genuinely ramified covers and \'etale covers,
the covers genuinely ramified in codimension $1$ and quasi-\'etale
covers can be seen as the basic building blocks from which we obtain all covers:

\begin{remark}
\label{rmk: factoring quasi-etale and genuinely ramified in codimension 1}
    Recall that a cover $Y\to X$ of a normal variety $X$
    can be factored via $Y\to X'\to X$, where $Y\to X'$ is a genuinely
    ramified cover and $X'\to X$ is an \'etale cover
    as the composition of two \'etale covers is \'etale. 
    This has an analogue for covers genuinely ramified in codimension $1$
    and quasi-\'etale covers as the composition of quasi-\'etale covers is still
    quasi-\'etale.
\end{remark}

We provide an example of a cover genuinely ramified in codimension $1$
by adapting the construction in \cite[Lemma 3.5]{fun}.

\begin{lemma}
\label{example: cyclic cover genuinely ramified in codimension 1}   
    Let $X$ be a normal projective variety of dimension at least $1$
    and $\mathcal{L}$ a reflexive sheaf of rank $1$ on $X$. 
    Fix an integer $d>1, p\nmid d,$ and assume for $1\neq e\mid d$
    that $\mathcal{L}$ does not admit an $e$-th root on every big open
    where $\mathcal{L}$ is a line bundle.
    Then there exists a cyclic cover $Y\to X$ of order $d$ which is genuinely
    ramified in codimension $1$
    and such that $\mathcal{L}_{\mid Y}$ admits a $d$-th root on
    some big open of $Y$ 
\end{lemma}
\begin{proof}
    The case where $X$ is a curve coincides with the setting of \cite[Lemma 3.5]{fun}.
    We assume that $X$ dimension at least $2$ in the following.

    Let $U$ be the big open locus where $\mathcal{L}$ is a line bundle.
    For $N\gg 0$ we have that $\mathcal{L}(Nd)_{\mid U}$ is very ample on $U$
    by \cite[\href{https://stacks.math.columbia.edu/tag/0FVC}{Tag 0FVC}]{sp}.
    Then the general section $s\in H^0(U,\mathcal{L}(Nd))$ gives rise to an effective
    Cartier divisor $D$ on $U$ by Bertini's theorem, 
    \cite[Theoreme 6.10, 3)]{jou}.    
    Say, $D$ is cut out by the dual $s^{\lor}$ which on an affine open neighborhood
    $\Spec(A)\subseteq U$
    intersecting $D$ corresponds to a non-zero divisor $f\in A$.

    Consider the field extension of $L/\kappa(X)=Q(A)$ generated by a $d$-th
    root of $f$. Then $\pi:Y\to X$ is a finite cyclic
    cover of order $d$, where $Y$ is the integral closure of $X$ in $L$,
    as by assumptions there are no roots of $\mathcal{L}$ on every big open where
    it is a line bundle
    (and thus no roots of $f$) and the characteristic does not divide $d$.
    Then by construction $\mathcal{L}(Nd)$ admits a $d$-th root
    on the preimage $\pi^{-1}(U)$ which is a big open in $Y$.

    An intermediate Galois cover $Y\to Y'\to X$ corresponds to the intermediate
    field extension of $L/\kappa(X)$ defined via taking an $e$-th root of $f$.
    As the ramification of $Y'\to X$
    contains the divisor cut out by $f$ on $\Spec(A)$
    we obtain that $Y\to X$ is genuinely ramified in codimension $1$.
\end{proof}

Every genuinely ramified Galois cover with smooth target
is connected in codimension $1$;
this was also observed in \cite[Section 6]{stratified-iso}.
\begin{cor}\label{codim1conn}
    Let $Y\to X$ genuinely ramified Galois cover of normal varieties 
    and let $X$ be smooth.
    Then $Y\times_X Y$ is connected in codimension at most $1$.
\end{cor}

\begin{proof}
    If $X$ is smooth, then the purity of branch locus,
    \cite[\href{https://stacks.math.columbia.edu/tag/0BMB}{Tag 0BMB}]{sp}, 
    implies that $Y\to X$ is ramified purely in codimension 1. 
    Thus, $Y\to X$ is genuinely ramified if and only if $G_1=G$. 
    Now apply Lemma \ref{lemma-connecting-group}.
\end{proof}

\subsection{Genuinely ramified in codimension 1 and stability: the Galois case}

From now on we focus on covers genuinely ramified in codimension $1$.
In this subsection we show that reflexive pullback by a Galois cover
genuinely ramified in codimension $1$ preserves stability,
i.e., we prove \Cref{thm:bdp analogue stable} in the Galois case.

We need the following generalization of \cite[Proposition 3.5]{bp}.
In the case where both $Y$ and $X$ are smooth this strengthening was also observed 
in \cite[Lemma 6.4]{stratified-iso}.

\begin{lemma}
\label{lemma-mod-torsion}
    Let $\pi:Y\to X$ be a genuinely ramified Galois cover of degree $n$ of
    normal varieties. Then
    \begin{align*}
        \pi^*((\pi_*\mathcal{O}_Y)/\mathcal{O}_X)/(\mathrm{torsion})\subset\bigoplus_{i=1}^{n-1}{\mathcal{L}_i},
    \end{align*}
    for some proper subsheaves $\mathcal{L}_i\subset\mathcal{O}_Y$.
    
    Moreover, if $Y\to X$ is in addition genuinely ramified in codimension $1$, 
    then $\supp(\mathcal{O}_Y/\mathcal{L}_i)$ has codimension $1$ in $Y$.
\end{lemma}
\begin{proof}
    Consider the cartesian diagram
    \[
    \begin{tikzcd}
        Y\times_X Y \ar[r,"\pr_2"] \ar[d,"\pr_1"] & Y \ar[d,"\pi"]\\
        Y \ar[r,"\pi"] & X.
    \end{tikzcd}
    \]
    We begin by showing
    $\pi^{\ast}\pi_{\ast}\mathcal{O}_Y/T=
    \pr_{1,\ast}\mathcal{O}_{(Y\times_X Y)_{red}}$, 
    where $T$ denotes the torsion submodule of
    $\pi^{\ast}\pi_{\ast}\mathcal{O}_Y$.
    
    Observe that
    $\pi^{\ast}\pi_{\ast}\mathcal{O}_Y=\pr_{1,\ast}\mathcal{O}_{Y\times_X Y}$
    by affine base change. 
    Let $G$ be the Galois group of $Y\to X$ and recall
    that the irreducible components of $Y\times_{X} Y$ are in one-to-one 
    correspondence to elements of $G$ and 
    equipped with the reduced subscheme structure 
    they are all isomorphic to $Y$ via $\pr_1$.
    
    As the intersection of the minimal primes of a ring is the nilpotent radical, we have
    $\mathcal{O}_{(Y\times_X Y)_{red}}\subseteq \bigoplus
    i_{\ast}\mathcal{O}_Y$, where the direct sum is taken over all irreducible
    components $i:Y\to Y\times_X Y$ equipped with the reduced subscheme structure.
    Thus, $\pr_{1,\ast}\mathcal{O}_{(Y\times_X Y)_{red}}$
    is torsion free over $\mathcal{O}_Y$ and we obtain a surjection 
    \[
    \pi^{\ast}\pi_{\ast}\mathcal{O}_Y/T\twoheadrightarrow \pr_{1,\ast}\mathcal{O}_{(Y\times_X Y)_{red}}.
    \]
    Being an isomorphism is a Zariski-local property and
    we can assume that $X$ (resp. $Y$) is the spectrum of a normal domain $A$ (resp. $B$).
    Then all that remains to show is that the nilradical $\eta$ of $B\otimes_A B$
    is torsion considered as a module over $B\to B\otimes_A B, b\mapsto b\otimes 1$.
    We claim that $\eta$ is contained in the kernel $K$ of
    \[
    B\otimes_A B\to Q(B)\otimes_{Q(A)}(Q(A)\otimes_A B)\subseteq Q(B)\otimes_{Q(A)}Q(B).
    \]
    Indeed, $Q(B)\otimes_{Q(A)}Q(B)$ is reduced as $Q(B)/Q(A)$ is Galois.
    Thus, $\eta$ is torsion as $K$ is torsion.

    We now continue with the proof of the lemma in the case where $Y\to X$
    is genuinely ramified in codimension $1$; equivalently 
    $Y\times_X Y$ is connected in codimension $1$, see \Cref{lemma-connecting-group}.
    The irreducible components of $Y\times_X Y$ can be labeled
    by $Y_1,\dots,Y_n$, where $n=\#G$, together 
    with an order preserving map 
    $\eta:\{1,\dots n\} \to \{0,\dots,n-1\}$
    and such that $Y_i\cap Y_{\eta(i)}$ has codimension $1$ as 
    a closed subscheme of $Y_i$ for $1<i\leq n$ as follows:
    Every irreducible component has a shortest path to the diagonal $Y_{e_G}$,
    where we set the distance of irreducible components to $1$ if they are distinct and intersect in codimension $1$.
    Then the labeling and $\eta$ are defined inductively as 
    \begin{itemize}
        \item $Y_1:=Y_{e_G}$ is defined to be the diagonal, $\eta(1):=0$,
        \item for $1<i\leq n$ we define $Y_i$ to be an unlabeled component of 
        the shortest distance among unlabeled components to $Y_1$ 
        and choose $\eta(i)$ such that $Y_{\eta(i)}$ has distance
        $1$ to $Y_i$ and shorter distance to $Y_1$.
    \end{itemize}

    Using this labeling we obtain
    $(\pi^{\ast}\pi_{\ast}\mathcal{O}_Y)/T\subseteq \bigoplus_{i=1}^n \mathcal{O}_{Y_i}\cong \bigoplus_{i=1}^{n}\mathcal{O}_Y.$
    Consider the morphism
    $\varphi:\bigoplus_{i=1}^n \mathcal{O}_Y \to 
    \bigoplus_{i=2}^{n} \mathcal{O}_Y, (s_i)\mapsto (s_i-s_{\eta(i)})$. 
    Then the kernel of $\varphi$ is given by
    \[
        \mathcal{O}_Y\subseteq 
        (\pi^{\ast}\pi_{\ast}\mathcal{O}_Y)/T\cong
        \pr_{1,\ast}\mathcal{O}_{(Y\times_X Y)_{red}}
        \subseteq \bigoplus_{i=1}^{n}\mathcal{O}_{Y},
    \]
    where the first inclusion is via adjunction. Thus
    \[
        (\pi^{\ast}\pi_{\ast}\mathcal{O}_Y)/(T+\mathcal{O}_Y)\subseteq \bigoplus_{i=2}^n\mathcal{O}_Y.
    \]
    
    Let $\mathcal{L}_i$ be the ideal sheaf cutting out $(Y_i\cap Y_{\eta(i)})_{red}$ 
    in $Y_i$. Projecting to the $i$-th component we claim that 
    \[
        (\pi^{\ast}\pi_{\ast}\mathcal{O}_Y)/(T+\mathcal{O}_Y) \to \bigoplus_{i=2}^n \mathcal{O}_Y\to \mathcal{O}_Y \to \mathcal{O}_Y/\mathcal{L}_i
    \]
    is trivial, i.e., 
    \[
    (\pi^{\ast}\pi_{\ast}\mathcal{O}_Y)/(T+\mathcal{O}_Y)\subseteq \bigoplus_{i=2}^n \mathcal{L}_i.  
    \]
    Indeed, a local section $s$ of $\pi^{\ast}\pi_{\ast}\mathcal{O}_Y\cong \pr_{1,\ast}\mathcal{O}_{Y\times_X Y}$
    satisfies $s_i=s_{\eta(i)}$ at $y_i$ for every closed point
    $y_i$ in $(Y_i\cap Y_{\eta(i)})_{red}$.
    This implies the claim as $Y_i\cap Y_{\eta(i)}$ is Jacobson.
    
    Since the $\mathcal{L}_i$ are ideal sheaves cutting out the closed subsets 
    $Y_{i}\cap Y_{\eta(i)}$ as codimension $1$ in $Y_i$ by construction, 
    it follows that the quotients 
    $\mathcal{O}_Y/\mathcal{L}_i$ are supported in codimension $1$.
    The case where $Y\to X$ is only genuinely ramified instead of genuinely
    ramified in codimension $1$ is analogous. The only change required is in the
    ordering of the irreducible components of $Y\times_X Y$: we only require
    non-empty intersection instead of codimension $1$ intersection.
\end{proof}

In the untwisted vector bundle setting and for a genuinely ramified
Galois cover of smooth varieties the following lemma was proven
in \cite[Lemma 6.4]{stratified-iso}.

\begin{lemma}
\label{lemma-homs}
    Let $\pi:Y\to X$ be a Galois cover of normal projective varieties.
    Assume that $\pi$ is genuinely ramified in codimension $1$.
    Let $F$ be a vector bundle on $X$.
    Let $\mathcal{V}_1$ and $\mathcal{V}_2$ 
    be semistable $F$-twisted reflexive sheaves on $X$ of the same slope.
    Then the natural morphism
    \[    
        \Hom_{F-TS}(\mathcal{V}_1,\mathcal{V}_2)\to \Hom_{F_{\mid Y}-TS}({\mathcal{V}_1}_{\mid Y},{\mathcal{V}_2}_{\mid Y})
    \]
    is an isomorphism.

    The same holds if $\mathcal{V}_1,\mathcal{V}_2$ are stable $F$-twisted
    vector bundles of the same slope
    and $Y\to X$ is a genuinely ramified Galois cover.
\end{lemma}
    
\begin{proof}
    Let $U$ be a big open in $X$ on which $\mathcal{V}_1$ and $\mathcal{V}_2$
    are vector bundles and $\pi$ is flat.
    Then we have an isomorphism
    \[
        \Hom_{F-TS}(\mathcal{V}_1,\mathcal{V}_2)\xrightarrow{\sim} \Hom_{F_{\mid U}-TS.}(\mathcal{V}_{1\mid U},\mathcal{V}_{2\mid U}),\;\; \varphi\mapsto \varphi_{\mid U}
    \]
    as $\mathcal{V}_1,\mathcal{V}_2$ are reflexive.
    We also have an analogous statement for 
    $V:=\pi^{-1}(U)$ and ${\mathcal{V}_i}_{\mid Y}$, $i\in\{1,2\}$.
    Thus, it suffices to show that the natural morphism
    \[
        \Hom_{F_{\mid U}-TS}(\mathcal{V}_{1\mid U},\mathcal{V}_{2\mid U})\to 
        \Hom_{F_{\mid V}-TS}(\pi_{\mid U}^{\ast}\mathcal{V}_{1\mid U},
        \pi_{\mid U}^{\ast}\mathcal{V}_{2 \mid U})
    \]
    is an isomorphism.
    Using adjunction we obtain
    \begin{align*}
        \Hom_{F_{\mid U}-TS}(\mathcal{V}_{1\mid U},\mathcal{V}_{2\mid U}) \to & \Hom_{F_{\mid V}-TS}(\pi_{\mid U}^{\ast}\mathcal{V}_{1\mid U},\pi_{\mid U}^{\ast}\mathcal{V}_{2\mid U})\\
                     & \cong \Hom_{F_{\mid U}-TS}(\mathcal{V}_{1\mid U},\pi_{\mid U,\ast}\pi_{\mid U}^{\ast}\mathcal{V}_{2\mid U}).
        \end{align*}
    Note that this composition is also induced by the short exact sequence
    \[
    0\to \mathcal{O}_X \to \pi_{\ast}\mathcal{O}_Y \to \pi_{\ast}\mathcal{O}_Y / \mathcal{O}_X \to 0
    \]
    via restricting to $U$, tensoring with $\mathcal{V}_{2\mid U}$, 
    and applying $\Hom_{F_{\mid U}-TS}(\mathcal{V}_{1\mid U},-)$.
    Thus, it suffices to show the vanishing of $\Hom_{F_{\mid U}-TS}(\mathcal{V}_{1\mid U}, \mathcal{V}_{2\mid U}\otimes \pi_{\mid U,\ast}(\mathcal{O}_V/\mathcal{O}_U)).$

    Let $T$ be the torsion submodule of $\pi^{\ast}\pi_{\ast}\mathcal{O}_Y$.
    By Lemma \ref{lemma-mod-torsion} we have 
    \[
        (\pi^{\ast}\pi_{\ast}\mathcal{O}_Y)/(\mathcal{O}_Y+T)\subseteq 
        \bigoplus_{i=1}^{d-1}\mathcal{L}_i
    \]
    for some subsheaves $\mathcal{L}_i\subseteq \mathcal{O}_Y$ such that $\mathcal{O}_Y/\mathcal{L}_i$ has support in codimension $1$, where $d$ 
    denotes the degree of $Y\to X$.
    Restricting to $V$ we obtain
    \[
        \pi_{\mid U}^{\ast}\pi_{\mid U,\ast}\mathcal{O}_V/(\mathcal{O}_V + T_{\mid V})\subseteq \bigoplus_{i=1}^{d-1}\mathcal{L}_{i\mid V}.
    \]
    As $\pi_{\ast}\mathcal{O}_Y/\mathcal{O}_X$ is torsion-free, 
    we find the inclusion 
    \[
        \pi_{\mid U,\ast}\mathcal{O}_V/\mathcal{O}_U\subseteq 
        \pi_{\mid U,\ast}(\pi_{\mid U}^{\ast}\pi_{\mid U,\ast}\mathcal{O}_V/(\mathcal{O}_V + T_{\mid V}))\subseteq 
        \bigoplus_{i=1}^{d-1}\pi_{\mid U,\ast}\mathcal{L}_{i\mid V}.
    \]
    Then we conclude the desired vanishing via computing
    \begin{align*}
        \ & \Hom_{F_{\mid U}-TS}(\mathcal{V}_{1\mid U},
        \mathcal{V}_{2\mid U}\otimes (\pi_{U,\ast}\mathcal{O}_V/\mathcal{O}_U))\\
        \ & \subseteq
        \Hom_{F_{\mid U}-TS}(\mathcal{V}_{1\mid U},
            \mathcal{V}_{2\mid U}\otimes
            \bigoplus_{i=1}^{d-1} \pi_{\mid U,\ast}\mathcal{L}_{i\mid V})\\
        \ &  = \bigoplus_{i=1}^{d-1}
            \Hom_{F_{\mid U}-TS}(\mathcal{V}_{1\mid U},
            \pi_{U,\ast}(\pi_{\mid U}^{\ast}\mathcal{V}_{2\mid U}
            \otimes (\mathcal{L}_i)_{\mid V}))\\ 
        \ & = \bigoplus_{i=1}^{d-1} \Hom_{F_{\mid V}-TS}
            (\pi_{\mid U}^{\ast}(\mathcal{V}_{1\mid U}),
            \pi_{\mid U}^{\ast}(\mathcal{V}_{2\mid U})\otimes (\mathcal{L}_i)_{\mid V})
         = 0,
    \end{align*}
    where in the first equality we use that
    $\mathcal{V}_{2\mid U}$ is a vector bundle.
    In the last equality, we use that $V$ is a big open,
    $\mu({\mathcal{V}_1}_{\mid Y})<\mu({\mathcal{V}_2}_{\mid Y}\otimes \mathcal{L}_i)$
    as $\mathcal{O}_Y/\mathcal{L}_i$ is supported in codimension $1$, 
    and that both ${\mathcal{V}_1}_{\mid Y}$ and ${\mathcal{V}_2}_{\mid Y}$
    are semistable of the same slope.

    The proof of the stable
    vector bundle case is analogous, except for the vanishing
    of $\Hom_{F_{\mid Y}-TS}({\mathcal{V}_1}_{\mid Y},{\mathcal{V}_2}_{\mid Y}\otimes \mathcal{L}_i)$
    as we only know that $\mathcal{L}_i\subsetneq \mathcal{O}_Y$
    are proper subsheaves with no information of the codimension of the support
    of $\mathcal{O}_Y/\mathcal{L}_i$. Indeed, a non-zero morphism of stable
    $F$-twisted vector bundles of the same slope has to be an isomorphism,
    see \Cref{lemma-stable-simple}.
    As ${\mathcal{V}_i}_{\mid Y}$ are both polystable $F_{\mid Y}$-twisted bundles of the same slope for $i\in\{1,2\}$,
    only the $0$-morphism 
    ${\mathcal{V}_1}_{\mid Y}\to{\mathcal{V}_2}_{\mid Y}$
    can factor via ${\mathcal{V}_2}_{\mid Y}\otimes \mathcal{L}_i$.
\end{proof}

\begin{cor}
    \label{cor: bdp analogue stable galois}
    Reflexive pullback by a Galois cover $Y\to X$ of normal projective varieties
    genuinely ramified in codimension $1$
    preserves stability of reflexive $F$-twisted sheaves,
    where $F$ is a vector bundle on $X$.
\end{cor}

\begin{proof}
    Recall that the reflexive pullback $\mathcal{V}_{\mid Y}$
    of a stable $F$-twisted sheaf $\mathcal{V}$ 
    is polystable. By \Cref{lemma-homs} we know the endomorphisms of $\mathcal{V}_{\mid Y}$.
    Indeed, we have 
    $\Hom_{F_{\mid Y}-TS}(\mathcal{V}_{\mid Y},\mathcal{V}_{\mid Y})
    =\Hom_{F-TS}(\mathcal{V},\mathcal{V})=k$
    since $\mathcal{V}$ is stable, see \Cref{lemma-stable-simple}.
    Thus, we conclude that the polystable reflexive twisted sheaf
    $\mathcal{V}_{\mid Y}$ has only one stable direct summand,
    i.e., it is stable.
\end{proof}

\subsection{Genuinely ramified in codimension 1 and stability: the general case}
\begin{proof}[Proof of Theorem \ref{thm:bdp analogue stable}]    
    Let $Y\to X$ be a cover of normal projective varieties
    genuinely ramified in codimension $1$.
    Let $F$ be a vector bundle on $X$.
    Consider a stable $F$-twisted sheaf $\mathcal{V}$ on $X$. 
    We wish to show that $\mathcal{V}_{\mid Y}$ is stable as an
    $F_{\mid Y}$-twisted sheaf.
    
    Let $Z\to Y\to X$ be the Galois closure of $Y\to X$,
    i.e., $Z$ is the normal closure of $X$ in the Galois closure of 
    the function fields $\kappa(Y)/\kappa(X)$.
    
    Consider the factorization of $Z\to X$ into 
    a cover $Z\to X'$ genuinely ramified in codimension $1$ 
    and a quasi-\'etale cover $X'\to X$.
    Note that $Z\to X'$ is Galois and so is 
    $X'\to X$ as the maximal quasi-\'etale part of a Galois cover.
    Denote the Galois group of $X'/X$ by $G$.
    
    Note that as $Y\to X$ is genuinely ramified in codimension $1$, 
    the normalization of the reduced fiber product $Y':=(Y\times_X X')_{red}$ 
    is a Galois cover with Galois group $G$,
    see \Cref{thm:bdp analogue surjective on fundamental group}.
    
    There is a morphism $Z\to Y'$ induced by $Z\to Y$ and $Z\to X'$ -
    in a picture
    \[
    \begin{tikzcd}
        Z \ar[ddr, bend right=30] \ar[rrd, bend left=30] \ar[dr, dashed] & \ & \\
        \ & Y' \ar[r] \ar[d] & X' \ar[d]\\
        \ & Y \ar[r] & X.
    \end{tikzcd}
    \]
    Then $\mathcal{V}_{\mid X'}$ is polystable.
    Consider the decomposition 
    $\mathcal{V}_{\mid X'}\cong\bigoplus_{i\in I} \mathcal{V}'_i$
    into stable $F_{\mid X'}$-twisted reflexive sheaves on $X'$.
    By the Galois case, \Cref{cor: bdp analogue stable galois},
    $\mathcal{V}'_{i\mid Z}$ is stable and then so is 
    $\mathcal{W}'_i:=\mathcal{V}'_{i\mid Y'}$.

    Consider a saturated $F$-twisted stable subsheaf 
    $\mathcal{W}\subset \mathcal{V}_{\mid Y}$ of slope 
    $\mu(\mathcal{W})=\mu(\mathcal{V}_{\mid Y})$ on $Y$.
    Then $\mathcal{W}_{\mid Y'}\subset \mathcal{V}_{\mid Y'}$ 
    is saturated and polystable of slope $\mu(\mathcal{V}_{\mid Y'})$.
    Thus, there is some subset $J\subset I$ such that 
    $\mathcal{W}_{\mid Y'}\cong\bigoplus_{j\in J}\mathcal{W}'_j$. 
    Identifying $\mathcal{V}'_i$ with its image in $\mathcal{V}_{\mid X'}$ 
    we find some $\mathcal{W}'_i\subset \mathcal{V}_{\mid Y'}$ 
    contained in $\mathcal{W}_{\mid Y'}$.
    
    Then on the one hand, the pullback of the image $\mathcal{E}$ of 
    \[
        \bigoplus_{\sigma\in G}\sigma^{\ast}\mathcal{V}'_i
        \xrightarrow{\oplus \varphi_{\sigma}^{-1}\sigma^{\ast}\iota} \mathcal{V}_{\mid X'}
    \]
    to $Y'$ is contained in $\mathcal{W}_{\mid Y'}$, 
    where $\iota:\mathcal{V}'_i\to \mathcal{V}_{\mid X'}$ denotes the inclusion
    and $\varphi_{\sigma}$ the $G$-linearization associated to $\mathcal{V}$.
    On the other hand, $\mathcal{E}$ is a $G$-invariant 
    saturated $F_{\mid X'}$-twisted subsheaf of $\mathcal{V}_{\mid X'}$ 
    and descends to a saturated $F$-twisted 
    subsheaf of $\mathcal{V}$ of the same slope, 
    see also the proof of \Cref{lemma-key}.
    Thus, $\mathcal{E}=\mathcal{V}_{\mid X'}$, as $\mathcal{V}$ is stable 
    and therefore $\mathcal{W}_{\mid Y'}=\mathcal{V}_{\mid Y'}$.
    We conclude $\mathcal{W}=\mathcal{V}_{\mid Y}$, i.e., $\mathcal{V}_{\mid Y}$ is stable.
\end{proof}

\section{Functorial notions of stability}

While semistability is preserved under pullback along a cover,
the same is not true for stability. 
In this section we show that the property
of a twisted reflexive sheaf on a smooth projective variety
(or twisted vector bundle on a normal projective variety)
to remain stable on all covers 
of degree prime to the characteristic $p\geq 0$ is an open property.
We do this via constructing an \'etale  prime to $p$ cover 
that checks for this property.
The construction of the cover checking for stability
also has an analogue for stable twisted reflexive sheaves
but "\'etale" has to be replaced by "quasi-\'etale".
Let us make the definition and statement precise:

\begin{definition}
    We call a cover $Y\to X$ of a normal projective variety $X$
    \emph{prime to $p$}, where $p$ is the characteristic of 
    the base field $k$, if the degree of the Galois closure $Y'\to X$ of $Y\to X$
    is not divisible by $p$.
\end{definition}

\begin{definition}
    Let $X$ be a normal projective variety.
    Let $F$ be a vector bundle on $X$.
    We call an $F$-twisted reflexive sheaf $\mathcal{V}$
    \emph{prime to $p$ stable} if for every quasi-\'etale prime to $p$ cover $Y\to X$
    we have that $\mathcal{V}_{\mid Y}$ is a stable $F_{\mid Y}$-twisted reflexive
    sheaf.
\end{definition}

\begin{remark}
    Note that it suffices to check prime to $p$ stability
    of a (twisted) bundle $\mathcal{V}$ on all \'etale prime to $p$ covers
    instead of all quasi-\'etale prime to $p$ covers.
    Indeed, only the maximal \'etale part of 
    a (quasi-\'etale) cover matters for the
    stability of the pullback of $\mathcal{V}$
    by \Cref{thm:bdp analogue stable}.
\end{remark}

\begin{theorem}[Analogue of Theorem 1, \cite{fun}]
    \label{thm:good cover}
    Let $X$ be a normal projective variety. Let $r\geq 2$ and $F$ be a 
    vector bundle on $X$.
    
    Then there exists a prime to $p$ quasi-\'etale Galois cover
    $X_{r,\reflgood}\to X$ such that 
    a stable $F$-twisted 
    reflexive sheaf $\mathcal{V}$ of rank $r$ on $X$
    is prime to $p$ stable if and only if 
    $\mathcal{V}_{\mid X_{r,\reflgood}}$ is stable as an 
    $F_{\mid X_{r,\reflgood}}$-twisted reflexive sheaf.

    Furthermore, there exists a prime to to $p$ \'etale Galois cover
    $X_{r,\good}\to X$ such that a stable $F$-twisted vector bundle
    $\mathcal{V}$ of rank $r$ on $X$
    is prime to $p$ stable if and only if 
    $\mathcal{V}_{\mid X_{r,\good}}$ is stable as
    an $F_{\mid X_{r,\good}}$-twisted bundle.

    If $X$ is in addition smooth, then we can arrange $X_{r,\good}=X_{r,\reflgood}$.
\end{theorem}

Since stability is an open property and pullback induces
a morphism of the level of stacks we obtain the following:

\begin{cor}\label{stability bundles open}
    Let $F$ be a vector bundle on a normal projective variety $X$.
    Then prime to $p$ stability of $F$-twisted bundles 
    on a normal projective variety is an open property.
\end{cor}

Similarly, using that for a smooth variety quasi-\'etale
covers coincide with \'etale covers, and thus reflexive pullback
is just pullback, we obtain:

\begin{cor}\label{stability refl open}
    Let $F$ be a vector bundle on a smooth projective variety $X$.
    Then prime to $p$ stability of $F$-twisted reflexive sheaves
    is an open property. 
\end{cor}

We split the construction of $X_{r,\reflgood}\to X$ into two parts.
By the key observation, \Cref{lemma-key}, a stable $F$-twisted reflexive sheaf
$\mathcal{V}$ decomposes on a quasi-\'etale Galois cover $Y\to X$ as
a direct sum of pairwise non-isomorphic
$F_{\mid Y}$-twisted stable reflexive sheaves 
$\mathcal{V}_{\mid Y}\cong \bigoplus_{i=1}^n \mathcal{W}_i^{\oplus e}$
such that the Galois group $\mathrm{Gal}(Y/X)$ acts transitively on the isomorphism
classes of the $\mathcal{W}_i$.
Then $\mathcal{W}_1^{\oplus e}$ descends as a direct summand of twisted reflexive sheaves 
to $Y\to Y/\mathrm{Stab}(\mathcal{W}_1)$ which is a quasi-\'etale cover of
degree $n$ over $X$, where $\mathrm{Stab}(\mathcal{W}_1)\subset \gal(Y/X)$
denotes the stabilizer of $\mathcal{W}_1$.
We obtain the following analogue of \cite[Lemma 3.3]{fun}:

\begin{lemma}
    \label{lemma:n at least 2}
    Let $Y\to X$ be a quasi-\'etale Galois cover of a normal projective variety $X$. 
    Let $F$ be a vector bundle on $X$ and 
    $\mathcal{V}$ be a stable $F$-twisted reflexive sheaf.
    If the decomposition
    $\mathcal{V}_{\mid Y}\cong \bigoplus_{i=1}^n \mathcal{W}_i^{\oplus e}$
    of \Cref{lemma-key} satisfies $n\geq 2$,
    then $\mathcal{V}\cong \pi'_{\ast}\mathcal{W'}$,
    where $Y\to X$ factors via the quasi-\'etale cover $Y\to Y'\xrightarrow{\pi'} X$, 
    $\pi'$ is of degree $n$, and $\mathcal{W'}$ is a stable $F_{\mid Y'}$ twisted sheaf
    on $Y'$.
\end{lemma}

Then defining $X_{r,\refllarge}$ as 
a quasi-\'etale Galois cover (of degree prime to $p$)
dominating all the finitely many quasi-\'etale (prime to $p$) covers
of $X$ of degree dividing $r$ we obtain the following:

\begin{lemma}
    \label{lemma:big cover}
    Let $X$ be a normal projective variety. Let $r\geq 2$.
    Let $F$ be a vector bundle on $X$.    
    There exist a quasi-\'etale Galois cover $X_{r,\refllarge}\to X$
    and an quasi-\'etale prime-to-$p$ Galois cover $X'_{r,\refllarge}\to X$
    satisfying the following: 
    
    If an $F$-twisted stable reflexive sheaf
    $\mathcal{V}$ of rank $r$ is stable after reflexive pullback to 
    $X_{r,\refllarge}$ (resp. $X'_{r,\refllarge}$),
    then for all Galois covers $Y\to X$ 
    (resp. prime-to-$p$ Galois covers $Y\to X$) 
    we have that $\mathcal{V}_{\mid Y}\cong \mathcal{W}^{\oplus e}$
    for some $e\mid r$ and a stable $F_{\mid Y}$-twisted
    reflexive sheaf $\mathcal{W}$ on $Y$.
\end{lemma}

Now, in the vector bundles case the proof of \Cref{thm:good cover} 
is analogous to the construction
of the good cover $X_{r,\good}\to X$
given in \cite{fun}.
In particular, the good covers can be chosen to be the same
and depend only on the rank $r$ and the variety $X$. 
To also obtain a version for reflexive sheaves,
we have to make several adjustments but the strategy remains the same.

\begin{definition}[Determinant]
    Let $\mathcal{V}$ be a reflexive sheaf of rank $r\geq 1$ on a normal variety $X$.
    The \emph{determinant} $\det(V)$ is defined as the reflexive hull of
    $\wedge^r \mathcal{V}$.
\end{definition}

\begin{proof}[Proof of \Cref{thm:good cover}]
    We assume that $X$ has dimension at least $1$; 
    the case where $X$ is a point being trivial.
    Assume we have constructed the cover $X_{r,\reflgood}\to X$.
    Then we claim that the maximal \'etale part $X_{r,\good}\to X$
    of $X_{r,\reflgood}\to X$
    is a prime to $p$ \'etale Galois cover and has the desired properties.
    Indeed, the genuinely ramified cover $X_{r,\reflgood}\to X_{r,\good}$ 
    preserves stability of $F$-twisted bundles
    under pullback by \Cref{thm:bdp analogue stable}.
    Furthermore, if $X$ is smooth, then every quasi-\'etale cover of $X$
    is \'etale by purity of branch locus. 
    Thus, we find that we can choose $X_{r,\good}=X_{r,\reflgood}$
    in the smooth case.
    We now continue with the remaining reflexive case and the construction of
    $X_{r,\reflgood}\to X$.
    
    Let $X_{r,\reflgood}\to X$ be a quasi-\'etale prime to $p$ Galois cover 
    dominating all the quasi-\'etale prime to $p$ covers of degree at most $rJ(r)$, 
    where $J(r)$ denotes the constant
    from Jordan's theorem in characteristic $0$ and the constant from
    the positive characteristic analogue due to  Brauer and Feit, 
    see \cite[p. 114]{jor} and \cite{jordanp} respectively.
    There are only finitely many such covers by \Cref{cor: quasi-etale of fixed degree}.
    
    One implication in the statement of the theorem is immediate.
    For the other implication assume that $\mathcal{V}$ is stable
    after pullback to $X_{r,\reflgood}$.
    Let $\pi:Y\to X$ be a quasi-\'etale prime to $p$ Galois cover of $X$
    with Galois group $G$. 
    
    By \Cref{lemma:big cover} we have that 
    $\mathcal{V}_{\mid Y}\cong \mathcal{W}^{\oplus e}$
    for a stable $F_{\mid Y}$-twisted reflexive sheaf $\mathcal{W}$.
    Furthermore, $\mathcal{W}$ is a $G$-invariant reflexive $F_{\mid Y}$-twisted
    sheaf, that is for $\sigma\in G$ there is an isomorphism
    $\mathcal{W}\cong \sigma^{\ast}\mathcal{W}$ of twisted sheaves.
    Note that this need not imply descend of $\mathcal{W}$ to $X$.
    It is however fairly close to descending: up to a finite cyclic prime to
    $p$ cover of $X'\to X$ and $Y'\to Y$ of degree dividing $r$ the determinant $\det(\mathcal{W})$ descends to $X$ over a big open.
    To be precise we claim that there exists a commutative diagram of prime to
    $p$ covers
    \begin{equation}
    \label{diagram:descending det}
    \begin{tikzcd}
        Y' \ar[d]\ar[r] & X' \ar[d]\\
        Y \ar[r] & X
    \end{tikzcd}
    \end{equation}
    such that we have the following:
    \begin{itemize}
        \item $X'\to X$ is a cyclic of degree dividing $r$, 
        \item $X'\to X$ is a cover genuinely ramified in codimension $1$,
        \item $Y'\to X'$ is a  quasi-\'etale Galois cover,
        \item
            $\mathcal{W}_{\mid Y'}\cong \mathcal{W}'^{\oplus e'}$ 
            for a stable $F_{\mid Y'}$-twisted reflexive sheaf $\mathcal{W}'$,
            and 
        \item $\det(\mathcal{W}')$ descends to $X'$ over a big open of $X'$.
    \end{itemize} 
    Note that if $\mathcal{W}_{\mid Y'}$
    is stable as an $F_{\mid Y'}$-twisted sheaf, then so is $\mathcal{W}_{\mid Y}$.
    We postpone the construction of diagram (\ref{diagram:descending det})
    to the end of the proof.
    
    The cover $Y'\to X'$ only has finitely many intermediate normal covers
    as these correspond to intermediate field extensions
    of the finite extension $\kappa(Y')/\kappa(X')$.
    Thus, we claim that there is a complete intersection curve $C'\hookrightarrow X'$
    such that for all intermediate covers $Y'\to Y''\to X'$
    we have 
    \begin{itemize}
        \item $D'':=C'\times_{Y''}X'$ is a smooth projective curve,
        \item $F_{\mid D''}$ and $\mathcal{W}_{\mid D''}$ are 
            vector bundles,
        \item $\mathcal{W}_{\mid Y''}$ is stable as an $F_{\mid Y''}$-twisted sheaf 
            if and only if $\mathcal{W}_{\mid D''}$ is stable as an 
            $F_{\mid D''}$-twisted bundle,
        \item $\mathcal{V}_{\mid Y''}$ is stable as an $F_{\mid Y''}$-twisted sheaf 
            if and only if $\mathcal{V}_{\mid D''}$ is stable as an 
            $F_{\mid D''}$-twisted bundle,
        \item $D''\to C'$ is \'etale, and 
        \item $\det(\mathcal{W}_{\mid D'})$ descends to $C'$, where $D':=Y'\times_{X'}C'$.
    \end{itemize}
    Indeed, this is a combination of the normal Bertini theorem, \cite[Theorem 7]{sei},
    the irreducible Bertini theorem, \cite[Corollaire 6.11 (3)]{jou},
    and the restriction theorem for stable twisted reflexive sheaves,
    \cite[Theorem 3.9]{langer-restriction-theorem} (which also holds for our convention
    of twisted sheaves),
    and the fact that the general hyperplane section of an ample bundle
    intersects a given closed subvariety properly (here: the locus where 
    the quasi-\'etale morphism $Y''\to X'$ is not \'etale, the locus where
    $\mathcal{V}_{\mid X'}$ is not a vector bundle,
    the complement of the locus to which $\det(\mathcal{W})$ descends; 
    all of these loci have codimension at least $2$).
    
    It remains to show that $\mathcal{V}_{\mid D'}$ 
    is a stable $F_{\mid D'}$-twisted bundle,
    where $D'\to C'$ is an \'etale prime to $p$ Galois cover with
    Galois group $\gal(D'/C')=\gal(Y'/X')$. 
    Note that $\mathcal{V}_{\mid C'}$ is stable and 
    $\mathcal{V}_{\mid D'}\cong \mathcal{W}_{\mid D'}^{\oplus e'}$ is the decomposition
    of \Cref{lemma-key} into stable direct summands. 
    Recall that $\mathcal{W}_{\mid D'}$ is simple as a twisted bundle.
    For $\sigma\in \gal(D'/C')$ we have an isomorphism 
    $\mathcal{W}_{\mid D'}\cong \sigma^{\ast}\mathcal{W}_{\mid D'}$ of twisted bundles
    which induces an isomorphism
    $\Hom_{F-TS}(\mathcal{W}_{\mid D'}, \sigma^{\ast}\mathcal{W}_{\mid D'})\cong k^{\ast}$.
    Taking determinants corresponds to taking $r$-th powers
    and we obtain a surjective morphism
    $\Hom_{F-TS}(\mathcal{W}_{\mid D'}, \sigma^{\ast}\mathcal{W}_{\mid D'})
    \twoheadrightarrow
    \Hom(\det(\mathcal{W}_{\mid D'}),\det(\sigma^{\ast}\mathcal{W}_{\mid D'}))$.
    The line bundle $\det(\mathcal{W}_{\mid D'})$
    descends to $C'$ by construction and we can lift the descend datum
    of $\det(\mathcal{W}_{\mid D'})$
    to a system of isomorphisms
    $\psi_{\sigma}:\mathcal{W}_{\mid D'}\xrightarrow{\sim} 
    \sigma^{\ast}\mathcal{W}_{\mid D'}$
    of twisted bundles.

    Consider the obstruction for the system of isomorphisms $\psi_{\sigma}$
    to be a descend datum $\lambda:=(\lambda_{\sigma,\tau})_{\sigma,\tau\in G}$,
    where 
    \[
        \lambda_{\sigma,\tau}:=
        \psi_{\sigma\tau}^{-1}\circ \tau^*(\psi_{\sigma})\circ \psi_{\tau}
        \in \Aut_{F_{\mid D'}-TS}(\mathcal{W}_{\mid D'})\cong k^*.
    \]
    Since $\psi_{\sigma}$ was a lift of a descend datum under taking
    determinants, it follows that $\lambda_{\sigma,\tau}^{r}$ is trivial.
    A computation verifies that 
    $(\lambda_{\sigma,\tau})_{\sigma,\tau\in G}$
    defines an inhomogeneous $2$-cocycle in $H^2(\gal(D'/C'),\mu_{r})$, 
    where $\gal(D'/C')$ acts trivially on $\mu_{r}$,
    see \cite[Proposition 2.8]{dk}.
    By \cite[Lemma 3.5]{fun} there exists a finite cyclic cover
    $\varphi:E'\to D'$ such that 
    \begin{itemize}
        \item $\deg(\varphi)\mid r$,
        \item $E'\to C'$ is Galois and prime to $p$,
        \item $\gal(E'/D')\subset \gal(E'/C')$ is central,
        \item $\varphi^*(\lambda)$ is a boundary of an inhomogenous $1$-cocycle
        $\alpha:\gal(E'/C')\to \mu_r$.
    \end{itemize}
    Thus, $\psi_{\sigma'}:=\varphi^*(\psi_{\sigma})\alpha(\sigma')^{-1}:
    \mathcal{W}_{\mid E'}
    \xrightarrow{\sim} \sigma'^{\ast}\mathcal{W}_{\mid E'}$
    defines a $\gal(E'/C')$-linearization,
    where $\sigma'\in\gal(E'/C')$ maps to $\sigma\in \gal(D'/C')$.

    Fix an isomorphism 
    $\psi:\mathcal{V}_{\mid D'}\xrightarrow{\sim}\mathcal{W}_{\mid D'}^{\oplus e}$
    of $F_{\mid D'}$-twisted bundles on $D'$ and let 
    $\psi^{\mathcal{V}}_{\sigma}:\mathcal{V}_{\mid D'}\xrightarrow{\sim}
    \sigma^{\ast}\mathcal{V}_{\mid D'}, \sigma\in \gal(D'/C')$ 
    be the descend datum associated to $\mathcal{V}$.
    Then the failure of the diagram
    \begin{equation}
        \begin{tikzcd}
        \label{representation diagram}
            \mathcal{V}_{\mid D'} \ar[r,"\psi"] \ar[d,"\psi_{\sigma}^{\mathcal{V}}"] &
            \mathcal{W}_{\mid D'}^{\oplus e} \ar[d,"\psi_{\sigma}^{\oplus e}"]\\
            \sigma^*\mathcal{V}_{\mid D'} \ar[r,"\sigma^*\psi"] &
            \sigma^{\ast}\mathcal{W}_{\mid D'}^{\oplus e}
        \end{tikzcd}
    \end{equation}
    to commute induces a representation 
    \[
        \rho:\gal(E'/C')\to \gl_e, \sigma'\mapsto
        \text{diag}(\alpha(\sigma')) \big{(}\psi \circ (\psi^{\mathcal{V}}_{\sigma})^{-1}
        \circ \sigma^{*}\psi^{-1} \circ \psi_{\sigma}^{\oplus e}\big{)},
    \]
    where $\text{diag}(\alpha(\sigma'))$ denotes the diagonal $e\times e$ matrix
    with $\alpha(\sigma')$ on each diagonal entry and $\sigma'$ maps to 
    $\sigma\in\gal(D'/C')$,
    see the proof of \cite[Lemma 3.7]{fun} for the computation in the untwisted setting
    which carries over verbatim.

    Apply Jordan's theorem (or the positive characteristic analogue of
    Brauer and Feit if $p>0$) 
    to obtain a normal abelian subgroup $N$ of the image
    $\im(\rho)$ of index at most $J(e)$, where $J(e)$ denotes Jordan's constant.
    Then $\rho^{-1}(N)+\gal(E'/D')$ is a normal subgroup of $\gal(E'/C')$
    which furthermore has abelian image under $\rho$
    since $\gal(E'/D')$ is central in $\gal(E'/C')$.
    By construction we have a factorization of Galois covers
    $E'/\ker(\rho)\to D'':=E'/(\rho^{-1}(N)+\gal(E'/D'))\to C'$,
    where the first cover is abelian and the second one is
    of degree at most $J(e)$. Furthermore, 
    $D'\to C'$ factors via $D''$,
    as we modded out by $\gal(E'/D')$.

    Recall that $D'\to C'$ was obtained from restricting the quasi-\'etale
    cover Galois cover $Y'\to X'$ to the large complete intersection curve $C'$
    with the same Galois group $\gal(D'/C')=\gal(Y'/X')$.
    In particular, $D''=Y''\times_{X'}C'$ for a unique intermediate
    quasi-\'etale Galois cover $Y'\to Y''\to X'$, namely the Galois
    closure of $X'$ in the intermediate field extension
    of $\kappa(Y')/\kappa(X')$ corresponding to the normal subgroup
    $\gal(D'/D'')\subset \gal(D'/C')$.
    Then $Y''\to X'\to X$ has degree at most $rJ(e)\leq rJ(r)$.
    Thus, $X_{r,\reflgood}\to X$ dominates 
    the maximal quasi-\'etale part of $Y''\to X$. 
    By assumption $\mathcal{V}_{\mid X_{r,\reflgood}}$ is a stable twisted
    sheaf and then so is $\mathcal{V}_{\mid Y''}$ by \Cref{thm:bdp analogue stable}.
    By construction of $C'$, this also implies that
    $\mathcal{V}_{\mid D''}$ is a stable twisted bundle.

    The pullback of the diagram (\ref{representation diagram}) to $E'$ and modifying
    it by $\text{diag}(\alpha(\sigma'))$
    \[
        \begin{tikzcd}
            \mathcal{V}_{\mid E'} \ar[r,"\psi"] \ar[d,"\psi_{\sigma}^{\mathcal{V}}"] &
            \mathcal{W}_{\mid E'}^{\oplus e} 
            \ar[d,"\text{diag}(\alpha(\sigma'))\varphi^*(\psi_{\sigma}^{\oplus e})"]\\
            \sigma'^*\mathcal{V}_{\mid E'} \ar[r,"\sigma'^*\psi"] &
            \sigma'^{\ast}\mathcal{W}_{\mid E'}^{\oplus e}
        \end{tikzcd}
    \]
    commutes for elements $\sigma'$ in the kernel of $\rho$ mapping
    to $\sigma\in\gal(D'/C')$ by definition of $\rho$.
    Restricting $\rho$ to $\gal(E'/D'')$ yields an irreducible representation
    $\rho':\gal(E'/D'')\to \gl_e$
    since $\mathcal{V}_{\mid D''}$ is stable as a twisted bundle.
    Furthermore, $E'/\ker(\rho)\to D''$ is abelian by construction of $D''$.
    As an abelian subgroup of $\gl_e$ is simultaneously triagonalizable,
    we conclude that $\rho'$ can only be irreducible if $e=1$.
    That is, we conclude that $\mathcal{V}_{\mid Y'}\cong \mathcal{W}$
    is stable.
    
    It remains to construct the commutative diagram \eqref{diagram:descending det}
    with the desired properties.
    Let $r'$ be the prime to $p$ part of $r$.
    By \Cref{example: cyclic cover genuinely ramified in codimension 1}
    applied to taking an $r'$-th root of the reflexive hull of a 
    root of $\det(\mathcal{W})$ defined over the big open where $\det(\mathcal{W})$
    is a line bundle,
    there exists a cyclic prime to $p$
    cover $\varphi:X'\to X$ of degree dividing $r'$ such that
    $\det(\mathcal{W})_{\mid X'}$ has an $r'$-th root over a big open $V'$ of $X'$,
    i.e., there exists a line bundle $\mathcal{L}$ on $V'$
    such that $\mathcal{L}^{\otimes r'}\cong \det(\mathcal{W})_{\mid V'}$.
    Furthermore, $X'\to X$ is genuinely ramified in codimension $1$.
    
    Consider the Galois closure $Y''$ over $X$ of the 
    normalization of an irreducible component of $Y\times_X X'$
    equipped with the reduced subscheme structure surjecting onto $X$.
    Then $Y''\to X$ is a prime to $p$ Galois cover factoring via $X'\to X$
    and $Y\to X$. Furthermore, $Y''\to X'$ is quasi-\'etale since $Y\to X$ is so.
    
    As $\mathcal{V}$ is stable after pullback along $X'_{r,\refllarge}\to X$, 
    we find $\mathcal{V}_{\mid Y''}\cong \mathcal{W}''^{\oplus e''}$
    for a stable reflexive sheaf $\mathcal{W}''$ on $Y''$.
    Note that $\det(\mathcal{V}_{\mid Y''})\cong \det(\mathcal{W}'')^{{\otimes e''}}$.
    The tensor product
    $\bigotimes_{\sigma\in \gal(Y''/X')}\sigma^{\ast}\det(\mathcal{W''})$
    descends to $X'$ over a big open and then so does $\det(\mathcal{W})''^{\otimes d}$,
    where $d:=\mathrm{gcd}(e'',\#\gal(Y''/X'))$.
    Note that $d$ is prime to $p$ since $Y''\to X'$
    is prime to $p$ and $d\mid r'$.
    
    By construction of $X'\to X$ we have an $r'$-th root $\mathcal{L}$
    of $\det(\mathcal{V}_{\mid X'})$
    on a big open $U'$ of $X'$. Let $V'$ be the preimage of $U'$ under $Y''\to X'$
    intersected with the big open locus where $\det(\mathcal{W}'')$ is a line bundle.
    Then $(\mathcal{L}^{\otimes -r'/d}\otimes \det(\mathcal{W}''))_{\mid V'}$
    is a line bundle on $V'$ whose $d$-th tensor power is trivial.
    Thus, there is a quasi-\'etale cyclic cover $Y'\to Y''$ of degree dividing $d$
    such that $(\mathcal{L}^{\otimes r'/d}_{\mid Y'})=\det(\mathcal{W}'')_{\mid Y'}$
    as reflexive rank $1$ sheaves, that is, we found descend
    of $\det(\mathcal{W}'')_{\mid Y'}$ over a big open.

    If $\mathcal{W''}_{\mid Y'}$ remains stable, then we have found the desired diagram.
    If not, then $\mathcal{V}_{\mid Y'}\cong \mathcal{W}'^{\oplus e'}$
    for some $e'>e''$ and we repeat the construction replacing $Y$ by $Y'$.
    Since, the number $e'$ is bounded by $r$ this process terminates.
\end{proof}

\section{Decomposition strata}
In this section, we define a stratification
of the moduli space of stable twisted Higgs bundles
via their behaviour under pullback
and study the image of the Hitchin morphism of the smallest stratum.
As an application, we show that the conjecture of Chen-Ngô 
(\cite[Conjecture 5.2]{chenngo}) holds for abelian and hyperelliptic varieties
(\Cref{hyperelliptic}).

\subsection{Moduli}

Throughout this section $X$ is a normal projective variety over an algebraically
closed field $k$, and $F$ is a vector bundle on $X$. 

We recall the definition of an $F$-twisted Higgs sheaf as well as 
some preliminiaries on moduli of Higgs-sheaves.

\begin{definition}
    An \emph{$F$-twisted Higgs sheaf} on $X$ is a pair $(\mathcal{V},\theta)$, where $\mathcal{V}$
    is a coherent sheaf and $\theta:\mathcal{V}\to\mathcal{V}\otimes F$ is an $\mathcal{O}_X$-linear morphism with the additional requirement that the composition 
    \[
    \theta\wedge\theta: \mathcal{V}\xrightarrow{\theta}\mathcal{V}\otimes F\xrightarrow{\theta\otimes\id}\mathcal{V}\otimes F^{\otimes2}\xrightarrow{\id\otimes\wedge}\mathcal{V}\otimes\wedge^2F
    \]
    vanishes. The last arrow is the identity of $\mathcal{V}$ tensored with the
    natural surjection $F^{\otimes2}\to\wedge^2F$.

    We define $\theta^{i}:\mathcal{V}\to \mathcal{V}\otimes F^{\otimes i}$
    to be the composition
    \[
        \mathcal{V}\xrightarrow{\theta}
        \mathcal{V}\otimes F \xrightarrow{\theta\otimes \mathrm{id}_F} 
        \mathcal{V} \otimes F^{\otimes 2} 
        \xrightarrow{\theta \otimes \mathrm{id}_{F^{\otimes 2}}}\cdots 
        \xrightarrow{\theta\otimes \mathrm{id}_{F^{\otimes {i-1}}}} 
        \mathcal{V}\otimes F^{\otimes i} 
    \]
    for $i\geq 1$.

    We call an $F$-Higgs sheaf \emph{Gieseker-semistable} if
    it is Gieseker-semistable as an $F$-twisted sheaf.
    Similarly, we define Gieseker-stable, slope-stable, and slope-semistable.
    
    When $X$ is smooth, setting $F=\Omega^1_X$ recovers the usual notion of a Higgs sheaf.
    A locally free $F$-twisted Higgs sheaf is called a \emph{$F$-twisted Higgs bundle}.
\end{definition}

For brevity, we often suppress the 'twisted' and just write $F$-Higgs sheaf
or $F$-Higgs bundle. To obtain moduli of $F$-Higgs bundles we recall the notion
of Gieseker semistability.

\begin{definition}[Gieseker stability]
    Let $\mathcal{V}$ be a $F$-twisted sheaf on a normal projective variety $X$. We call $\mathcal{V}$ \emph{Gieseker semistable} if for any $F$-twisted subsheaf $0\neq\mathcal{W}\subset\mathcal{V}$, the inequality of reduced Hilbert polynomials $p_{\mathcal{W}}\le p_{\mathcal{V}}$ holds. 
    
    We call $\mathcal{V}$ \emph{Gieseker stable} if for every proper $F$-twisted subsheaf $0\neq\mathcal{W}\subsetneq \mathcal{V}$ the strict inequality $p_{\mathcal{W}}<p_{\mathcal{V}}$ holds.
\end{definition}

\begin{definition}[JH-filtration]
    Let $\mathcal{V}$ be a Gieseker semistable $F$-twisted sheaf on $X$. 
    Then $\mathcal{V}$ admits a filtration 
    $0=\mathcal{V}_0\subset\mathcal{V}_1\subset\dots\subset
    \mathcal{V}_m=\mathcal{V}$
    by $F$-twisted subsheaves such that the associated graded $\bigoplus_{i=0}^m\mathcal{V}_i/\mathcal{V}_{i-1}$ 
    is a Gieseker polystable $F$-twisted sheaf 
    (i.e. it is a direct sum of $F$-twisted sheaves 
    with the same reduced Hilbert polynomial). 
    This is called a \emph{Jordan-H\"older} (JH for short) filtration of $\mathcal{V}$.
\end{definition}

\begin{definition}[S-equivalence]
    Two $F$-twisted Gieseker semistable 
    sheaves $\mathcal{V}$ and $\mathcal{W}$ on $X$ are called \emph{S-equivalent} 
    if the associated graded objects $\bigoplus_i\mathcal{V}_i/\mathcal{V}_{i-1}$ 
    and $\bigoplus_j\mathcal{W}_j/\mathcal{W}_{j-1}$
    with respect to the JH-filtration are isomorphic as $F$-twisted sheaves.
\end{definition}

We denote by $M_X^{P,\FGss}$ the moduli space of torsion-free,
Gieseker-semistable, $F$-Higgs sheaves of 
Hilbert polynomial $P$ (see \cite{langer-properness} for the construction).
Further, we denote by $M^{P,F-s}_X$ 
the moduli space of slope-stable $F$-Higgs
bundles of Hilbert polynomial $P$ on $X$. 

Given any $F$-Higgs sheaf $(\mathcal{V},\theta)$, we can view $\theta\in
H^0(X,F\otimes\mathrm{End}(\mathcal{V}))$ as an $F$-valued endomorphism of
$\mathcal{V}$ and associate to it its characteristic polynomial. The
characteristic polynomial is a degree $r$ polynomial in a formal variable whose
degree $i$ coefficient is the trace $\tr\theta^{r-i}$ of $\theta^{r-i}$. Here,
$r$ denotes the rank of Higgs sheaves with Hilbert polynomial $P$. Although
$\tr\theta^i\in H^0(X,F^{\otimes i})$ a priori, the Higgs condition
$\theta\wedge\theta=0$ implies that we may view $\tr\theta^i$ as an element of
$H^0(X,\mathrm{Sym}^iF)$ for all $1\le i\le r$.

\begin{definition}
    The \emph{Hitchin morphism} 
    $h^{P,F}:M^{P,\FGss}_X\to\bigoplus_{i=1}^rH^0(X,\mathrm{Sym}^iF)$
    is given by sending $(\mathcal{V},\theta)$ to the coefficients
    $(\tr\theta,\tr\theta^2,\dots,\tr\theta^r)$ of its characteristic polynomial.
    Henceforth, we suppress the Hilbert polynomial $P$ 
    and the bundle $F$ and denote the Hitchin morphism by $h^r$.
    The target of the Hitchin morphism is called the \emph{Hitchin base}
    and we denote it by $A^{r,F}_X$. 
\end{definition}

By \cite[Theorem 3.8]{langer-properness},
we know that the Hitchin morphism is proper in our setting.
As a consequence, we observe in the next lemma that
taking direct sums of Higgs sheaves induces a finite morphism between moduli
spaces. The proof uses a projective completion of the Hitchin base, which is
inspired by the projective compactification of Dolbeault moduli spaces
constructed in \cite{decat}.

\begin{lemma}\label{direct-sum-finite}
    Let $X$ be a normal projective variety and
    $F$ be a vector bundle on $X$.
    Let $P\in \mathbf{Q}[x]$ be a polynomial with rational coefficients.
    Denote by $r$ the rank of torsion-free sheaves with Hilbert polynomial $P$.
    Assume $r\geq 2$ and $r=r_1+r_2$ be a sum of non-negative integers $r_1,r_2$.
    
    Then the direct sum morphism 
    \[
    \oplus: M^{\frac{r_1}{r}P ,\FGss}_X\times M^{\frac{r_2}{r}P,\FGss}_X\to
    M^{P,\FGss}_X,\;\;
    \]
    \[
        ((\mathcal{V}_1,\theta_1), (\mathcal{V}_2,\theta_2))\mapsto
    (\mathcal{V}_1\oplus\mathcal{V}_2,(\theta_1\otimes \id)\oplus(\id \otimes \theta_2))
    \]
    is finite.
\end{lemma}
\begin{proof}
    We first show that $\oplus$ is quasi-finite. As quasi-finiteness
    is an open property, it suffices to show that over every closed point
    the preimage under $\oplus$ is finite.
    
    Note that closed points of $M^{P,\FGss}_X$
    correspond to Gieseker-polystable $F$-twisted torsion-free Higgs sheaves 
    of rank $r$ with reduced Hilbert polynomial $p$,
    where $p$ is the reduced Hilbert polynomial associated to $P$.
    
    Now, given a Gieseker-polystable $F$-twisted torsion-free Higgs sheaf
    $\mathcal{V}$ of rank $r_1+r_2$
    and reduced Hilbert polynomial $p$, there are only finitely many choices
    (up to isomorphism) to write $\mathcal{V}$ as a direct sum 
    $\mathcal{V}_1 \oplus \mathcal{V}_2$
    of Gieseker-polystable $F$-twisted torsion-free Higgs sheaves 
    with reduced Hilbert polynomial $p$. Indeed, each $\mathcal{V}_i$ 
    is a direct sum of the finitely many Gieseker-stable 
    direct summands of $\mathcal{V}$ by 
    the uniqueness of the associated graded of the JH-filtration. We obtain the desired quasi-finiteness
    and it remains to show that the direct sum morphism is proper.

    The characteristic polynomial of 
    $(\mathcal{V}_1\oplus\mathcal{V}_2,\theta_1\oplus\theta_2)$
    is given by the product of the characteristic polynomials of 
    $(\mathcal{V}_1,\theta_1)$ and $(\mathcal{V}_2,\theta_2)$. 
    Thus, $\oplus$ induces a morphism
    $\oplus:A^{r_1,F}_X\times A^{r_2,F}_X\to A^{r_1+r_2,F}_X$, 
    $((s_1,\dots,s_{r_1}),(t_1,\dots,t_{r_2}))\mapsto(v_1,\dots,v_{r_1+r_2})$, 
    where $v_i$ is given by $(-1)^i$ times the degree $r_1+r_2-i$
    coefficient of the polynomial
    $(y^{r_1}-s_1y^{r_1-1}+\dots+(-1)^{r_1}s_{r_1})
    \cdot(y^{r_2}-t_1y^{r_2-1}+\dots+(-1)^{r_2}t_{r_2})$, 
    for $1\le i\le r_1+r_2$. 
    For example, 
    $v_1=s_1+t_1$, $v_2=s_2+s_1\cdot t_1+t_2$,...,$v_{r_1+r_2}=
    s_{r_1}\cdot t_{r_2}$.
    
    We have the following commutative diagram

\begin{equation}\label{csum}
\begin{tikzcd}
       M^{\frac{r_1}{r}P,\FGss}_X\times M^{\frac{r_2}{r}P,\FGss}_X 
       \ar{r}{\oplus} \ar{d}{h^{r_1}\times h^{r_2}} & M_X^{P,\FGss} 
       \ar{d}{h^{r}}\\
        A^{r_1,F}_X\times A^{r_2,F}_X \ar{r}{\oplus} & A^{r,F}_X.
\end{tikzcd}
\end{equation}

Note that $h^{r_1}\times h^{r_2}$ is proper because both $h^{r_1}$
and $h^{r_2}$ are proper.
We claim that $\oplus:A^{r_1,F}_X\times A^{r_2,F}_X\to A^{r,F}_X$
is proper as well.

Let $P^{r,F}_X$ denote the projective completion of $A^{F,r}_X$ formed 
by adding an extra variable in degree zero.
More precisely,
$P^{r,F}_X\cong(\bigoplus_{i=0}^rH^0(X,\mathrm{Sym}^iF)
\setminus\{\mathbf{0}\})\sslash\mathbb{G}_m$,
where $H^0(X,\mathrm{Sym}^0F)=H^0(X,\mathcal{O}_X)$, $\mathbf{0}
:=(0,\dots,0)\in\bigoplus_{i=0}^rH^0(X,\mathrm{Sym}^iF)$
is the "origin", and the $\mathbb{G}_m$-action is given by scaling,
i.e., $\lambda\cdot(s_0,\dots,s_n)=(\lambda s_0,\dots,\lambda s_n)$.
It follows from the construction that $P^{r,F}_X$
is a projective space over $k$.

Thus, a point in $P_X^{r,F}$ can be represented by 
$[s_0:\dots:s_n]$ in homogeneous coordinates
and $A_X^{r,F}$ can be identified with the chart $\{s_0\neq0\}$. 

We define a map $\overline{\oplus}:P^{r_1,F}_X\times P^{r_2,F}_X\to P^{r,F}_X$ 
extending $\oplus$, given by
\[
    ([s_0:s_1:\dots:s_{r_1}],[t_0:t_1:\dots:t_{r_2}])\mapsto[v_0:v_1:\dots:v_{r}],
\]
where each $v_i$ is $(-1)^i$ times the degree $r-i$ coefficient of the polynomial 
\[
    (s_0y^{r_1}-s_1y^{r_1-1}+\dots+(-1)^{r_1}s_{r_1})\cdot
    (t_0y^{r_2}-t_1y^{r_2-1}+\dots+(-1)^{r_2}t_{r_2}),
\]
for $0\le i\le r_1+r_2=r$. 

Note that $\overline{\oplus}$ is indeed well defined, because the product of two
such polynomials is identically zero if and only if one of them is identically zero.
This follows from the fact that the ring 
$\bigoplus_{i=1}^\infty H^0(\eta,\Sym^i F_{\eta})$ 
is a polynomial ring
over the generic point $\eta=\Spec(\kappa(X))$.
We recover $\oplus$ by restricting $\overline{\oplus}$ to the chart
$\{s_0\neq0\}\times\{t_0\neq0\}$.
We obtain the Cartesian diagram

\begin{center}
\begin{tikzcd}
       A^{r_1,F}_X\times A^{r_2,F}_X \ar[hookrightarrow]{r} \ar{d}{\oplus} & P^{r_1,F}_X\times P^{r_2,F}_X \ar{d}{\overline{\oplus}}\\
        A^{r,F}_X \ar[hookrightarrow]{r} & P^{r,F}_X
\end{tikzcd}
\end{center}
from which it follows that $\oplus:A^{r_1,F}_X\times A^{r_2,F}_X\to A^{r,F}_X$ is proper.
By the commutativity of \ref{csum}, we obtain the properness of 
\[
    \oplus: M^{\frac{r_1}{r}P,\FGss}_X\times M^{\frac{r_2}{r}P,\FGss}_X\to
    M^{P,\FGss}_X.
\]
From the properness together with the quasi-finiteness, we conclude that $\oplus$ is finite. 
\end{proof}

\subsection{The split cover and strata}
The splitting of a stable twisted Higgs bundle after pullback along 
an \'etale Galois cover enables us to construct a stratification on moduli spaces 
of stable Higgs bundles with respect to this cover.
Iterating the good cover constructed in \Cref{thm:good cover}
we obtain a stratification on these moduli spaces that is independent of the
cover.

\begin{definition}[Decomposition type with respect to a cover]\label{dec type}
    Let $\pi:Y\to X$ be an \'etale Galois cover with Galois group $G$
    of a normal projective variety $X$.
    Let $F$ be a vector bundle on $X$ and
    let $\mathcal{V}$ be a stable $F$-twisted bundle of rank $r$ on $X$.
    Then by \Cref{lemma-key} we have 
    $\mathcal{V}_{\mid Y}\cong \bigoplus_{i=1}^n \mathcal{W}_i^{\oplus e}$
    for $n,e\geq 1$ and stable $F_{\mid Y}$-twisted bundles $\mathcal{W}_i$
    such that $G$ acts transitively on the isomorphism classes of the $\mathcal{W}_i$.
    Thus by the transitivity, $m:=\rk(\mathcal{W}_i)$ is independent of $i$ and the
    same holds for $P':=P(\mathcal{W}_i)$ as the Hilbert polynomial is computed with
    respect to $\pi^{\ast}\mathcal{O}_X(1)$ which is invariant under $G$.

    We call the natural number $m$ the \emph{decomposition type} 
    of $\mathcal{V}$ with respect to $\pi$.
\end{definition}

As in \cite{strata}, the decomposition type 
allows us to define a stratification 
of the moduli space of $F$-twisted stable Higgs
bundles with respect to an \'etale Galois cover.

\begin{definition}[Decomposition stratification with respect to a cover]
For an \'etale Galois cover $\pi:Y\to X$ of a normal projective variety $X$
we define decomposition strata using Lemma \ref{lemma-key}. Let $P\in\mathbb{Q}[x]$ be a polynomial
and assume that the moduli space $M^{P,F-s}_X$ 
of $F$-twisted
stable Higgs bundles with Hilbert polynomial $P$ is non-empty. 
Denote the rank of bundles with Hilbert polynomial $P$ by $r$.
We define for integers $m\geq 1$ such that $m\mid r$ the following stratum

\[
    Z^{F-s,P}(m,\pi):=
    \left\{(V,\varphi)\in M^{P,F-s}_X 
    \middle \vert
    \begin{array}{c}
    (\mathcal{V},\varphi) \text{ has decomposition}\\
    \text{type }m \text{ with respect to }\pi
    \end{array}
    \right\}.
\]
\end{definition}

\begin{lemma}
    Let $\pi:Y\to X$ be an \'etale Galois cover of a normal projective variety $X$.
    Let $P$ be a polynomial with rational coefficients and $F$ be a vector bundle on $X$. 
    Let $r$ be the rank of torsion-free sheaves on $X$ with Hilbert polynomial $P$.
    
    Then the decomposition stratification with respect to $\pi$ 
    defines a stratification with ordering given by 
    \[
        Z^{F-s,P}(m',\pi)\leq Z^{F-s,P}(m,\pi):\Leftrightarrow m'\mid m
    \]
    in the sense that the strata are disjoint and
    $\bigcup_{m'\mid m}Z^{F-s,P}(m',\pi)$ is closed
    in $M^{P,F-s}_X$ for all $m\mid r$.
\end{lemma}

\begin{proof}
    It is clear that the strata are disjoint by definition.
    The interesting claim is that the union of all smaller strata is closed.
    Note that the moduli space of stable Higgs bundles
    with Hilbert polynomial $P$ is of finite type.
    In particular, there are only finitely many Hilbert polynomials
    $P'_1,\dots,P'_n$ that can appear as
    the Hilbert polynomial of a pullback $\mathcal{V}_{\mid Y}$
    of a stable Higgs bundle with Hilbert polynomial $P$.
    Observe that all the stable direct summands of $\mathcal{V}_{\mid Y}$
    have the same Hilbert polynomial since the Galois group of $Y\to X$
    acts transitively on them by the key lemma, \Cref{lemma-key},
    and the Hilbert polynomial is computed
    with respect to the pullback of the polarization of $X$ to $Y$
    which is invariant under this action as well.
    
    Then $\bigcup_{m'\mid m}Z^{F-s,P}(m',\pi)$
    is the union of the preimages of the image of the direct sum map
    \[
       \prod_{j=1}^{r/m} M^{\frac{m}{r}P'_i,F_{\mid Y}\mathrm{-G-ss}}_Y\to M^{P'_i,F_{\mid Y}\mathrm{-G-ss}}_X, 
    (\mathcal{W}_1,\dots,\mathcal{W}_{r/m})\mapsto \mathcal{W}_1\oplus \dots
    \oplus \mathcal{W}_{r/m}
    \]
    for $i=1,\dots,n$ under
    $\pi^{\ast}:M^{P,F-s}_X\to \bigsqcup_{i=1}^n M^{P'_i,F_{\mid Y}\mathrm{-G-ss}}_X$.
    Since direct sum induces a finite morphism by \Cref{direct-sum-finite},
    we obtain the desired closedness.
\end{proof}

\begin{lemma}[Split cover]\label{splitcover}
    Let $X$ be a normal projective variety. Let $r\geq 2$ be an integer
    and $F$ be a vector bundle on $X$.
    Then there exists an \'etale prime to $p$ Galois cover
    $X_{r,\split}\to X$ such that for every \'etale prime to $p$
    Galois cover $Y\to X$ dominating $X_{r,\split}\to X$
    any stable $F$-twisted vector bundle of rank $\leq r$
    on $X$ has the same decomposition
    type with respect to $Y$ as it does with respect to $X_{r,\split}$.
\end{lemma}

\begin{proof}
This is analogous to \cite[Lemma 4.2]{strata}:
By iterating the good cover we obtain a prime to $p$ \'etale Galois
cover $X_{r,\split}\to X$ such that an $F$-twisted stable bundle 
$\mathcal{V}$ of rank $r$ on $X$
satisfies that $\mathcal{V}_{\mid X_{r,\split}}$
is a direct sum of prime to $p$ stable 
$F_{\mid X_{r,\split}}$-twisted bundles on $X_{r,\split}$.
\end{proof}

\begin{definition}[Decomposition stratification]\label{abs-decomp-strat}
    We call the decomposition stratification of $M^{P,F-s}_X$
    with respect to $X_{r,\split}\to X$ 
    the \emph{prime to $p$ decomposition stratification}
    (or just the \emph{decomposition stratification}).
    We denote the strata for $m\mid r$ by
    \[
        Z^{F-s,P}(m):=Z^{F-s,P}(X_{r,\split}\to X,m).
    \]

    This is justified as the decomposition type
    of a twisted Higgs bundle $\mathcal{V}$ of rank $r$
    with respect to $X_{r,\split}\to X$
    agrees with the decomposition type of $\mathcal{V}$
    with respect to every prime to $p$ Galois cover 
    $Y\to X$ dominating $X_{r,\split}\to X$, see \Cref{splitcover}.

    When $m=1$, we call $Z^{F-s,P}(1)$ the \emph{1-stratum} of $M_X^{P,F-s}$.
\end{definition}    

\begin{definition}[Semistable $1$-stratum]
    Consider the moduli space $M^{P,\FGss}_X$ 
    of Gieseker semistable torsion-free $F$-Higgs sheaves with Hilbert
    polynomial $P$. 
    
    Define the \emph{semistable $1$-stratum} to be the 
    closed subset $Z^{\FGss,P}(1)$ of S-equivalence 
    classes of Gieseker semistable $F$-Higgs bundles in $M^{P,\FGss}_X$ 
    which split as a direct sum of $F_{\mid X_{r,\split}}$-Higgs bundles 
    of rank $1$ on the split cover $X_{r,\split}$.
\end{definition}

\subsection{The image of the semistable 1-stratum under the Hitchin morphism}
Let $X$ be a normal projective variety and let $F$ be a vector bundle on $X$. 
We interpret an element of the Hitchin base
$A_X^{r,F}=\bigoplus_{i=1}^rH^0(X,\mathrm{Sym}^iF)$
as a degree $r$ homogeneous polynomial in a formal variable $y$
as already done in \Cref{direct-sum-finite}.
In other words, we express
$s=(s_1,\dots,s_r)\in A_X^{r,F}$ as $y^r-s_1y^{r-1}+\dots+(-1)^rs_r$.

Multiplication of polynomials then induces maps 
\[
    \prod\nolimits_{p,q}:A_X^{p,F}\times A_X^{q,F}\to A_X^{r,F}
\]
for all $p,q\in\mathbb{Z}_{\ge1}$ such that $p+q=r$.
These maps are finite (\Cref{direct-sum-finite}).

\begin{definition}
Let $(p,q)\subset A_X^{r,F}$ denote the set-theoretic 
image of $\prod_{p,q}$, and define $(r)$
to be the subset $(r):=A_X^{r,F}\setminus\bigcup_{p\le q}(p,q)$. 
We obtain the Hitchin base as the union
\[
A_X^{r,F} = \bigcup_{p\le q}(p,q)\cup(r).
\]

Each of the subsets $(p,q)$ is closed because it is the image of a finite morphism. 
Thus, the subset $(r)$ is open. 
Points of $(p,q)$ are those polynomials of $A_X^{r,F}$
that split as a product of a degree $p$ and a degree $q$ polynomial, 
while points of $(r)$ are the irreducible polynomials.
\end{definition}

A special case is when we consider
the multiplication map 
\begin{align*}
\prod\nolimits_{1,\dots,1}:(A_X^{1,F})^r:=&A_X^{1,F}\times\dots\times A_X^{1,F}\to A_X^{r,F} \\
&{\big (}(y-s_1),\dots,(y-s_r){\big )}\mapsto (y-s_1)\cdots(y-s_r).
\end{align*}

The set-theoretic image is the closed subset of $A_X^{r,F}$ 
consisting of polynomials that split completely,
i.e. as a product of $r$ degree 1 polynomials. 

\begin{definition}
We call the image of the map $\prod_{1,\dots,1}$
the \emph{1-stratum} of the Hitchin base $A_X^{r,F}$ 
and denote it by $(1)^{\oplus r}$.
\end{definition}

Let $\pi:Y\to X$ be an \'etale Galois cover. 
Then pullback induces a 
morphism 
on the Hitchin bases $\pi^*:A_X^{r,F}\to A_Y^{r,F_{\mid Y}}$
which is injective on closed points.

Note that $\pi^*$ preserves each of the closed subsets $(p,q)$, 
and also the 1-stratum $(1)^{\oplus r}$. 
However, it does not necessarily preserve the subset $(r)$.
The pullback of an irreducible degree $r$ polynomial in $A_X^{r,F}$ may
decompose as a product of lower degree polynomials in $A_Y^{r,F_{\mid Y}}$.

\begin{definition}
    Let $Y\to X$ be an \'etale Galois cover.
    We define for $\pi:Y\to X$ the closed subset 
    $(p,q)_{Y\to X}:=(\pi^*)^{-1}(p,q)_Y$, 
    where $(p,q)_Y$ denotes the image of $\prod_{p,q}$
    in $A_Y^{r,F_{\mid Y}}$. 
    In other words, $(p,q)_{Y\to X}$ consists of those polynomials
    in $A_X^{r,F}$ that split as a product of degree $p$ and degree $q$
    polynomials after pullback to $Y$.
    We similarly define $(1)^{\oplus r}_{Y\to X}$.
\end{definition}

\begin{remark}
    The affine scheme $(A^{1,F}_X)^r$ admits a permutation action 
    by the symmetric group $S_r$.
    The Keel-Mori theorem 
    (\cite[\href{https://stacks.math.columbia.edu/tag/0DUT}{Tag 0DUT}]{sp})
    tells us that the quotient stack $[(A^{1,F}_X)^r/S_r]$ admits
    a coarse (i.e. $k$-points correspond to $S_r$-orbits) moduli space, 
    which is the GIT quotient $(A^{1,F}_X)^r\sslash S_r$. 
    
    We see from this that the $k$-points of $(A_X^{1,F})^r\sslash S_r$ can be
    identified with unordered tuples $[\omega_1,\dots,\omega_r]$,
    where $\omega_i\in H^0(X,F)$ for all $1\le i\le r$. 
    The next Lemma shows that each such unordered tuple can in turn be
    identified with a (totally split) degree $r$ polynomial
    $\prod_{i=1}^r(y-\omega_i)$.
\end{remark}

\begin{lemma}\label{uf}
Let $X$ be a normal projective variety and $F$ a vector bundle on $X$.
Let $s\in (1)^{\oplus r}$ be in the 1-stratum of the Hitchin base $A^{r,F}_X$.
Then the decomposition $s=\prod_{i=1}^n(y-s_i)$ 
is unique up to permutation of the $s_i$'s.
Further, the induced morphism
$\prod_{1,...,1}:(A_X^{1,F})^r\sslash S_r \to A_X^{r,F}$ is injective
on closed points.
\end{lemma}
\begin{proof}
    Roughly speaking, the Hitchin base $A^{r,F}_X$ is locally 
    a unique factorization domain (UFD).
    To be precise, after passing to the generic point $\eta=\Spec(\kappa(X))$
    we have that $\bigoplus_{i=1}^{\infty}H^0(\eta,\mathrm{Sym}^iF_{\eta})$ is a UFD
    since it is isomorphic to a polynomial ring over $\kappa(X)$ 
    in $\rk(F)$ many variables.
    Thus, restricting two given decompositions
    $s= \prod_{i=1}^r(y-s_i)=\prod_{i=1}^r(y-s'_i)$
    to the generic point, 
    we have that the $(s_i)\mid _{\eta}$
    are unique up to permutation.
    It follows that the decomposition on $X$ 
    is also unique up to permutation as the localization maps 
    $H^0(X,\Sym^l F)\to H^0(\eta,\Sym^l F|_\eta)$ 
    are injective for all $l$.

    To see the injectivity of
    $\prod_{1,...,1}:(A_X^{1,F})^r\sslash S_r \to A_X^{r,F}$
    consider a $k$-point of the source as an unordered tuple 
    $[(y-\omega_1),\dots,(y-\omega_r)]$, where $\omega_i$
    is a global section in $A_X^{1,F}=H^0(X,F)$
    for all $1\le i\le r$. Its image in the Hitchin base $A_X^{r,F}$ is the polynomial
    $(y-\omega_1)\cdots(y-\omega_r)$ and it follows from the first part that the decomposition
    $\prod_{j=1}^r(y-\omega_j)$ is unique up to permutation of the $\omega_j$'s.
\end{proof}

\begin{lemma}\label{roots-distinct}
Let $Y\to X$ be an \'etale prime to $p$ 
Galois cover of a normal projective variety with Galois group $G$.
Let $F$ be a vector bundle on $X$.
Further, let $s\in(1)^{\oplus r}_{Y\to X}\cap(r)\subset A^{r,F}_X$. 

Then $s_{\mid Y}=\prod_{i=1}^r(y-t_i)$, and $G$ acts
on the set $\{t_1,\dots,t_r\}$ transitively.
Furthermore, the $t_i$ are pairwise distinct.
\end{lemma}
\begin{proof}
    Note that $G=\mathrm{Aut}(Y/X)$. We consider
    each $\sigma\in G$ as
    an automorphism $\sigma:Y\to Y$ fixing $X$.
    Denote the $G$-linearization of $F_{\mid Y}$ by
    $\varphi^F_{\sigma}:F_{\mid Y}\xrightarrow{\sim}\sigma^{\ast}F_{\mid Y}$
    for $\sigma\in G$. 
    Similarly, denote the natural $G$-linearization of $\mathcal{O}_{Y}$
    by $\varphi^{\mathcal{O}_Y}_{\sigma}$.
    
    Given a global section $t\in H^0(Y,F_{\mid Y})$ we obtain a global section
    $\sigma\cdot t\in H^0(Y,F_{\mid Y})$ defined as the composition
    $\mathcal{O}_Y\xrightarrow{\varphi^{\mathcal{O}_Y}_{\sigma}}
    \sigma^{\ast}\mathcal{O}_Y \xrightarrow{\sigma^{\ast}t}
    \sigma^{\ast}F_{\mid Y}\xrightarrow{(\varphi^F_{\sigma})^{-1}}
    F_{\mid Y}.$
    Similarly, we define $\sigma\cdot t$ for global sections of
    symmetric powers of $F$.
    
    The pullback $s_{\mid Y}$ of $s\in (1)^{\oplus r}_{Y\to X}\cap (r)$
    defines a $G$-equivariant global section
    $\mathcal{O}_Y\xrightarrow{s_{\mid Y}}
    \bigoplus_{i=1}^r \Sym^i F_{\mid Y}$,
    that is, $s_{\mid Y}$
    is compatible with the $G$-linearizations on source and target.
    We can phrase this as $\sigma\cdot s_{\mid Y}=s_{\mid Y}$ in the
    above notation.

    Observe that 
    $\sigma\cdot s_{\mid Y} 
    = \sigma\cdot \prod_{i=1}^r(y-t_i)=\prod_{i=1}^r (y-\sigma \cdot t_i)$, for some $t_i\in H^0(Y,F_{\mid Y})$, for $1\le i\le r$.
    Then, by the uniqueness of the decomposition $s_{\mid Y}=\prod_{i=1}^r (y-t_i)$ by
    \Cref{uf},
    we find that $\sigma \cdot t_i=t_{\sigma(i)}$ for some index $\sigma(i)$.
    This induces an action of $G$ on $\{t_1,\dots,t_r\}$.
    
    Fix an index $l$ and consider the $G$-orbit $G\cdot t_l$ 
    consisting of $m_l$
    distinct elements.
    Let $\{t^1_l,\dots,t^{m_l}_l\}$ be the set of distinct elements of $G\cdot t_l$.
    Then the $j$-th symmetric polynomial in $t^1_l,\dots,t_l^{m_l}$ defines a
    $G$-equivariant global section of $\Sym^j(F_{\mid Y})$.
    Thus, all coefficients of $P=(y-t_l^1)\cdot(y-t_l^2)\cdots(y-t_l^{m_l})$
    descend to $X$. Denote the polynomial on $X$ obtained by $P'$.
    As $P$ divides $s_{\mid Y}$, we find that $P'$ divides $s$ 
    by the injectivity $H^0(X,\Sym^j F)\to H^0(Y,\Sym^j F_{\mid Y})$
    for all $j$.
    
    Since $s\in(r)$ is irreducible, we conclude that $P'=s$. 
    In particular, $m_l=r$ i.e., the action of $G$ on $\{t_1,\dots,t_r\}$ 
    is transitive and all $t_i$ are pairwise distinct by definition of $m_l$.
\end{proof}

\begin{lemma}\label{H-split}
    Let $X$ be a normal projective variety.  Let $r\geq 2$ be an integer and let $F$ be a vector bundle on $X$.
    Then there exists an \'etale prime to $p$ Galois cover 
    $X_{r,\Hsplit}\to X$
    such that for every \'etale prime to $p$ Galois cover $Y\to X$ 
    dominating $X_{r,\Hsplit}\to X$,
    we have that 
    $(1)^{\oplus r}_{Y\to X}=(1)^{\oplus r}_{X_{r,\Hsplit}\to X}$.
\end{lemma}

\begin{proof}
    Define $X_{r,\Hsplit}\to X$ to 
    be an \'etale prime to $p$ Galois
    cover dominating all \'etale prime to $p$
    covers of $X$ of degree $\leq r$.

    Let $Y\to X$ an \'etale prime to $p$ Galois cover
    dominating $X_{r,\Hsplit}\to X$.
    The inclusion 
    $(1)^{\oplus r}_{X_{r,\Hsplit}\to X}\subseteq (1)^{\oplus r}_{Y\to X}$
    is immediate.
    To see the converse, consider $s\in (1)^{\oplus r}_{Y\to X}$
    and let $s=\prod_{i=1}^n s_i$ be a decomposition
    of $s$ into irreducible polynomials of degree $r_i$.
    Then we have that $s_i\in (1)^{\oplus r_i}_{Y\to X}\cap (r_i)$
    in the Hitchin base $A^{r_i,F}_X$.
    Since $r_i\leq r$ we are reduced to showing the lemma
    for irreducible $s$.

    Consider the action of the Galois group $G:=\gal(Y/X)$
    on the pairwise distinct factors $t_i$ of 
    $s_{\mid Y}=\prod_{i=1}^r (y-t_i)$
    obtained from Lemma \ref{roots-distinct}.
    Let $H_i$ be the stabilizers of $t_i$ for $i=1,\dots,r$.
    Then the section $t_i$ is the pullback of a section $t'_i$ of $F_{\mid Y_i}$,
    where $Y\to Y_i:=Y/H_i\to X$ is the intermediate \'etale cover
    associated to $H_i\subset G$.
    The same holds for the elementary symmetric polynomials in
    the $t_j$ for $j\neq i$.
    We conclude that $(y-t'_i)$ is a factor of $s_{\mid Y_i}$. 

    Note that $Y_i\to X$ is an \'etale prime to $p$ cover
    of degree $\leq r$ for each $i$.
    Thus, the cover $X_{r,\Hsplit}\to X$ dominates 
    $Y_i\to X$ for all $1\le i\le r$.
    It follows that $s_{\mid X_{r,\Hsplit}}$ is
    divisible by $(y-(t'_i)_{\mid X_{r,\Hsplit}})$ for all $1\le i\le r$.
    Since the $t_i$ are pairwise distinct and
    $(t'_i)_{\mid X_{r,\Hsplit}}$ pulls back to $t_i$ on $Y$,
    the $(t'_i)_{\mid X_{r,\Hsplit}}$ 
    are pairwise distinct as well.
    Thus, $\prod_{i=1}^r (y-(t'_i)_{\mid X_{r,\Hsplit}})$ is
    a monic polynomial of degree $r$ dividing $s_{\mid X_{r,\Hsplit}}$. 
    Since $s_{\mid X_{r,\Hsplit}}$ is a pullback of $s$, 
    it is also a monic polynomial of degree $r$ and thus 
    $s_{\mid X_{r,\Hsplit}} = \prod_{i=1}^r  (y-(t'_i)_{\mid X_{r,\Hsplit}})$ i.e., 
    $s\in(1)^{\oplus r}_{X_{r,\Hsplit}\to X}$.
\end{proof}

\begin{lemma}
\label{lemma:descend invariant sections}
Let $\pi:Y\to X$ be an \'etale Galois cover of a normal projective variety with
Galois group $G$. Let $F$ be a vector bundle on $X$ and $r\geq 2$.
Let $s\in (r)\cap (1)^{\oplus r}_{Y\to X}$
and $s_{\mid Y}=\prod_{i=1}^r (y-t_i)$ for some $t_i\in H^0(Y,F_{\mid Y})$.

Then $s=h_X(\pi'_{\ast}t')$ for some intermediate \'etale cover $Y\to Y'\xrightarrow{\pi'}X$
of degree $r$ and $t'\in H^0(Y',F_{\mid Y'})$ such that $t'_{\mid Y}=t_i$ for some $i$,
where $\pi'_{\ast}t'$ is considered as an $F$-twisted Higgs bundle with underlying
vector bundle $\pi'_{\ast}\mathcal{O}_{Y'}$.
Furthermore, $\pi'_{\ast}t'$ is a stable $F$-twisted Higgs bundle.

\end{lemma}
\begin{proof}
    Let $H$ be the stabilizer of a fixed $t_i$ under the action of $G$ on
    $\{t_1,\dots,t_r\}$. Then the cover $Y\to X$ factors as $Y\to Y':=Y/H\xrightarrow{\pi'} X$.
    
    The section $t_i\in H^0(Y,F_{\mid Y})$ is $H$-equivariant
    by definition
    of $H$ and thus descends to a section $t'\in H^0(Y',F_{\mid Y'})$. 
    Consider via pushforward the $F$-twisted Higgs bundle
    $\pi'_{\ast}(t'):\pi'_{\ast}\mathcal{O}_{Y'}\to \pi'_{\ast}F_{\mid Y'}\cong
    \pi'_{\ast}\mathcal{O}_{Y'}\otimes F$, where
    the isomorphism follows from the projection formula.
    We claim that $\pi'_{\ast}(t')$ has characteristic polynomial $s$.

    First observe that $(\pi'_{\ast}\mathcal{O}_{Y'})_{\mid Y}\cong\mathcal{O}_Y^{\oplus r}$
    since $Y\to X$ is Galois and $Y'\to X$ has degree $r$.
    By adjunction we have a commutative diagram
    \[
        \begin{tikzcd}
            \pi'^{\ast}\pi'_{\ast}\mathcal{O}_{Y'}
            \ar[r,"\pi'^{\ast}\pi'_{\ast}(t')"] \ar[d,two heads] & 
            \pi'^{\ast}\pi'_{\ast}\mathcal{O}_{Y'}  \otimes F_{\mid Y'}
            \ar[d, two heads]\\
            \mathcal{O}_{Y'}\ar[r,"t'"] & F_{\mid Y'}
        \end{tikzcd}
    \]
    with surjective vertical arrows which further pulls back to
    \[
        \begin{tikzcd}
            \mathcal{O}_Y^{\oplus r} \ar[r,"\pi^{\ast}\pi'_{\ast}t'"] 
            \ar[d, two heads] &
            \mathcal{O}_Y^{\oplus r} \otimes F_{\mid Y} \ar[d, two heads]\\
            \mathcal{O}_Y \ar[r,"t_i"] & F_{\mid Y}.
        \end{tikzcd}
    \]
    Pulling back along $\sigma\in G$ yields an analogous diagram
    for all $t_1,\dots,t_r$ via the transitive action of $G$ on $\{t_1,\dots,t_r\}$.
    Since all $t_i$ are pairwise distinct we find that 
    the $F_{\mid Y}$-twisted Higgs bundle 
    $(\mathcal{O}_Y^{\oplus r},\pi^{\ast}\pi'_{\ast}t')$
    is $S$-equivalent
    to the $F_{\mid Y}$-twisted Higgs bundle $\bigoplus_{i=1}^r(\mathcal{O}_Y,t_i)$
    which by construction of the $t_i$ has characteristic polynomial $s_{\mid Y}$.
    By the injectivity of $H^0(X,\Sym^j F)\to H^0(Y,\Sym^j F_{\mid Y})$
    we conclude that the characteristic polynomial of $\pi'_{\ast}t'$ is $s$.

    Note that the underlying vector bundle of 
    $\pi'_{\ast}(t')$ is \'etale trivializable
    and in particular semistable.
    Since $s$ is by assumption irreducible we conclude that $\pi'_{\ast}(t')$
    is stable as an $F$-twisted Higgs bundle.
\end{proof}

For a smooth projective variety in characteristic $0$,
the conjectured set-theoretic image of the Hitchin morphism 
is a closed subscheme of the Hitchin base,
called the \emph{spectral base} (see \cite{chenngo} for details). 

In the smooth case, the $1$-stratum for the split cover $X_{r,\Hsplit}\to X$ is always contained in the image of the Hitchin morphism from the moduli space of Gieseker-semistable $F$-Higgs bundles with vanishing Chern classes. 

\begin{theorem}\label{1-stratum image}
    Let $X$ be a smooth projective variety and 
    let $F$ be a vector bundle on $X$.
    Let $P_0$ denote the Hilbert polynomial
    of the trivial vector bundle of rank $r$.
    The set-theoretic image of the semistable $1$-stratum 
    $Z^{\FGss,P_0}(1)\subset M^{P_0,\FGss}_X$
    via the Hitchin morphism $h^r:M^{P_0,\FGss}_X\to A^{r,F}_X$ 
    is precisely the $1$-stratum for the split cover 
    $(1)^{\oplus r}_{X_{r,\Hsplit}\to X}\subset A^{r,F}_X$.
\end{theorem}

 \begin{proof}
    Consider a closed point $s\in(1)^{\oplus r}_{X_{r,\Hsplit}\to X}$. 
    By induction, it suffices to treat the case that $s$ is irreducible, 
    i.e., that $s\in (1)^{\oplus r}_{X_{r,\Hsplit}\to X}\cap(r)$.

    Then it follows from Lemma \ref{lemma:descend invariant sections} 
    that there is a stable $F$-Higgs bundle $(\mathcal{V},\theta)$
    of rank $r$
    whose underlying vector bundle $\mathcal{V}$ is \'etale-trivializable, 
    and $s$ is the characteristic polynomial of $(\mathcal{V},\theta)$. 
    In particular, the Chern classes of $\mathcal{V}$ vanish, 
    so it lies in $Z^{\FGss,P_0}(1)\subset M^{P_0,\FGss}_X$.
\end{proof}

\begin{cor}\label{Fettriv}
    Let $X$ be a smooth projective variety and let $F$ be an \'etale trivializable vector bundle of rank $n$ on $X$. 
    Then the set-theoretic
    image of the Hitchin morphism 
    $h^r:M^{P_0,\FGss}_X\to A^{r,F}_X$ 
    coincides with the $1$-stratum $(1)^{\oplus r}_{X_{r,\Hsplit}\to X}$. 
\end{cor}
\begin{proof}
    By Theorem \ref{1-stratum image} we know that 
    $(1)^{\oplus r}_{X_{r,\Hsplit}\to X}$ is contained in the image of $h^r$.
    Thus, it suffices to show that $h^r$ factors through 
    $(1)^{\oplus r}_{X_{r,\Hsplit}\to X}$.

    Let $(\mathcal{V},\theta)\in M^{P_0,\FGss}_X$ be 
    a closed point and let
    $s:=h^r(\mathcal{V},\theta)\in A^{r,F}_X$ be the characteristic polynomial
    of $\theta$. Let $Y\to X$ be an \'etale cover dominating $X_{r,\Hsplit}\to
    X$ that trivializes $F$, i.e., $F_{\mid Y}\cong\mathcal{O}_Y^{\oplus n}$.
    Then $s_{\mid Y}\in A^{r,O_Y^{\oplus n}}_Y$ is the characteristic polynomial
    of $(\mathcal{V}_{\mid Y},\theta_{\mid Y})\in M^{P_0,\mathcal{O}_Y^{\oplus
    n}-\mathrm{G-ss}}_Y$.

    We want to show that $s_{\mid Y}\in(1)^{\oplus r}\subset A^{r,O_Y^{\oplus
    n}}_Y$. Note that the global sections of the bundle
    $\mathrm{Sym}^j\mathcal{O}_Y^{\oplus n}$ are constant for all $j\ge1$. The
    Higgs field $\theta_{\mid Y}$ is locally an 
    $r\times r$ matrix whose entries
    are local sections of $\mathcal{O}^{\oplus n}_Y$. The characteristic
    polynomial of $\theta$ is $s_{\mid Y}$ and locally, the roots of $s_{\mid
    Y}$ are precisely the eigenvalues of $\theta_{\mid Y}$.

    Since the coefficients of $s_{\mid Y}$ are constant, the eigenvalues of
    $\theta_{\mid Y}$ are locally constant, 
    hence constant on all of $Y$ because $Y$ is connected. 
    Thus, we have a splitting $s_{\mid Y}=\prod_{i=1}^r(y-t_i)$, 
    where $t_i\in H^0(Y,\mathcal{O}_Y^{\oplus n})$ for $1\le i\le r$. 
    It follows that $s_{\mid Y}\in(1)^{\oplus r}$ and thus 
    $s\in(1)^{\oplus r}_{Y\to X}$. 
    Since $Y\to X$ dominates $X_{r,\Hsplit}\to X$, 
    we have $s\in(1)^{\oplus r}_{X_{r,\Hsplit}\to X}$ by Lemma \ref{H-split}.
\end{proof}

As an application, we can determine the image of the Hitchin morphism from the \emph{Dolbeault moduli space} of Gieseker semistable $\Omega^1_X$-Higgs bundles with vanishing Chern classes, when $X$ is a hyperelliptic variety, that is, a
smooth projective quotient of an abelian variety.

\begin{cor}\label{hyperelliptic}
    Let $X$ be a smooth projective variety of dimension $d$ such that there is an \'etale Galois cover $Y\to X$, where $Y$ is an abelian variety. Then the
    set-theoretic
    image of $h^r:M^{P_0,\Omega^1_X-\mathrm{G-ss}}_X\to A^{r,\Omega^1_X}_X$ 
    coincides with the subset 
    $(H^0(Y,\Omega^1_Y)^r\sslash S_r)^G$
    of $A^{r,\Omega^1_Y}_Y$, where $G:=\gal(Y/X)$.

    In particular, if $X$ is an abelian variety, the set-theoretic image of $h^r$ is $(\mathbb{A}^d)^r\sslash S_r$.
\end{cor}
\begin{proof}
    Since $Y\to X$ is \'etale and $Y$ is an abelian variety, it follows that
    we have isomorphisms
    $(\Omega^1_X)_{\mid Y}\cong\Omega^1_Y\cong\mathcal{O}_Y^{\oplus d}$, i.e.,
    $\Omega^1_X$ is an \'etale trivializable vector bundle 
    (see e.g. \cite[Chapter 4(iii)]{mumford}).
    By Corollary \ref{Fettriv},
    the set-theoretic image of $h^r$ is equal to $(1)^{\oplus r}_{Y\to X}$.

    Note that the $G$-action on $H^0(Y,\Omega^1_Y)^r\sslash S_r$ 
    is precisely the $G$-action on the corresponding set of polynomials, i.e., 
    for any closed point $v:=[v_1,\dots,v_r]$ of $H^0(Y,\Omega^1_Y)^r\sslash S_r$ 
    and any $g\in G$, we have 
    \[
        g\cdot v=g\cdot[v_1,\dots, v_r]=g\cdot(\sigma_1(v),\dots\sigma_r(v))=(g\cdot\sigma_1(v),\dots,g\cdot\sigma_r(v)),
    \]
    where $\sigma_i(v)\in H^0(Y,\mathrm{Sym}^i\Omega^1_Y)$ 
    is the $i$-th elementary symmetric polynomial in $[v_1,\dots,v_r]$.

    By definition, $(1)^{\oplus r}_{Y\to X}$ contains all polynomials $s$ such that $s_{\mid Y}\in(1)^{\oplus r}\subset A^{r,\Omega^1_Y}_Y$. Recall that the pullback morphism $A^{r,\Omega^1_X}_X\to A^{r,\Omega^1_Y}_Y$ is injective
    on closed points, and $(1)^{\oplus r}\cong (A^{1,\Omega^1_Y}_Y)^r\sslash S_r\cong H^0(Y,\Omega^1_Y)^r\sslash S_r$. We can thus identify $s$ with $s_{\mid Y}$, and $s_{\mid Y}$ with its set of roots $\{t_i\}_{i=1}^r$, $t_i\in H^0(Y,\Omega^1_Y)$ for all $i$.
    Since $s_{\mid Y}$ is a pullback from $X$, it is $G$-equivariant. Thus, $s_{\mid Y}=[t_1,\dots,t_r]\in(H^0(Y,\Omega^1_Y)^r\sslash S_r)^G$.

    On the other hand, every characteristic polynomial $t$ in
    $A^{r,\Omega^1_Y}_Y$ splits as a product of degree $1$ factors, thus $t$ is
    a closed point of $H^0(Y,\Omega^1_Y)^r\sslash S_r$. 
    Suppose now that $t$ is a closed point of 
    $(H^0(Y,\Omega^1_Y)^r\sslash S_r)^G$,
    then it is $G$-equivariant, in particular,
    its coefficients $t_i$ are $G$-equivariant global sections of 
    $\mathrm{Sym}^i\Omega^1_Y$ and thus descend to global sections 
    $s_i$ of $\mathrm{Sym}^i\Omega^1_X$, for $1\le i\le r$.

    Thus, $t$ descends to a polynomial $s\in(1)^r_{Y\to X}\subset A^{r,\Omega^1_X}_X$,
    whose coefficients are $s_i$, for $1\le i\le r$, i.e., $t=s_{\mid Y}$.
    From Lemma \ref{lemma:descend invariant sections},
    it follows that $s\in\im(h^r)$, and we conclude that 
    $\im(h^r)=(H^0(Y,\Omega^1_Y)^r\sslash S_r)^G$ set-theoretically.
\end{proof}

\subsection{Specializing to classical Higgs bundles}
Let $X$ be a smooth projective variety over $\mathbb{C}$.
Then by Simpson's classical non-abelian Hodge correspondence 
(see \cite{simp,simp2}), there is a
homeomorphism between the Dolbeault moduli space $M^r_{X,\text{Dol}}$
of Gieseker semistable Higgs bundles of rank $r$
with vanishing Chern classes on $X$,
and the \emph{Betti moduli space} 
$M^r_{X,\mathrm{Betti}}$ of conjugacy classes of 
semisimple representations 
$\pi_1(X)\to GL(r,\mathbb{C})$ of the topological fundamental group.

\begin{remark}\label{bettistrat}
    The decomposition stratification on $M^r_{X,\text{Dol}}$ induces a
    decomposition stratification on $M^r_{X,\mathrm{Betti}}$ via the non-abelian
    Hodge correspondence. The decomposition type of an irreducible
    representation $\rho:\pi_1(X)\to GL(r,\mathbb{C})$ is defined analogously to
    that of a stable Higgs bundle, with respect to the split cover
    $X_{r,\split}\to X$. 
\end{remark}

\bibliographystyle{plain}

\bibliography{bibliography}

\end{document}